\documentclass[11pt]{article}

\usepackage[margin=1in]{geometry}
\usepackage{amsmath,amssymb,amsthm,mathtools}
\usepackage{comment}
\usepackage{microtype}
\usepackage{hyperref}
\usepackage{xcolor}
\usepackage{authblk}
\usepackage{mathrsfs}
\usepackage[backend=biber, style=numeric, maxnames=99]{biblatex} 
\hypersetup{
  colorlinks=true,
  linkcolor=red,
  urlcolor=red,
  citecolor=red
}

\newcommand{\KL}{\mathrm{KL}}
\newcommand{\KLinf}{\mathrm{KL}_{\mathrm{inf}}}
\newcommand{\infKL}{\textnormal{inf-KL}}
\newcommand{\EPow}{\operatorname{EPow}_n}
\newcommand{\E}{\mathbb{E}}
\newcommand{\Pcal}{\mathcal{P}}
\newcommand{\Qcal}{\mathcal{Q}}

\newcommand{\X}{\mathsf{X}}
\newcommand{\B}{\mathcal{B}}
\newcommand{\one}{\mathbf{1}}
\newcommand{\bPn}{(\mathcal{P}^n)^{\circ\circ}}
\newcommand{\bPk}{(\mathcal{P}^k)^{\circ\circ}}
\newcommand{\Bb}{\mathcal{B}_b}
\newcommand{\bbF}{\mathbb F}
\newcommand{\bbG}{\mathbb G}

\newcommand{\bt}{\mathbf t}
\newcommand{\Tdet}{\mathsf T}
\newcommand{\Wcal}{\mathcal W}

\DeclareMathOperator{\cl}{cl}

\newtheorem{theorem}{Theorem}[section]
\newtheorem{proposition}[theorem]{Proposition}
\newtheorem{lemma}[theorem]{Lemma}
\newtheorem{corollary}[theorem]{Corollary}
\theoremstyle{definition}
\newtheorem{definition}[theorem]{Definition}
\newtheorem{remark}[theorem]{Remark}
\newtheorem{example}[theorem]{Example}
\theoremstyle{definition}

\author[]{Ashwin Ram\thanks{aram2@andrew.cmu.edu} }
\author[]{Aaditya Ramdas\thanks{aramdas@cmu.edu}}
\affil[]{Carnegie Mellon University}

\title{The optimal betting wealth growth rate}
\date{\today}

\begin{document}

\maketitle

\begin{abstract}
This paper characterizes the best possible rate of growth of wealth in a Kelly betting game when repeatedly betting against a general i.i.d.\ null hypothesis $\Pcal$, but the data are drawn i.i.d.\ from an arbitrary alternative $Q$. We prove that it equals $\lim_{n \to \infty}n^{-1}\inf_{P \in (\Pcal^n)^{\circ\circ}} \KL(Q^n,P)$, where $\Pcal^n = \{P^n: P \in \Pcal\}$ and $(\Pcal^n)^{\circ\circ}$ is its bipolar, i.e.\ this rate is achievable and one cannot do better.  This quantity is in general smaller than a more popular quantity in the literature, $\KLinf(Q,\Pcal) := \inf_{P \in \Pcal}\KL(Q,P)$. If $\KLinf(\cdot,\Pcal)$ is weakly lowersemicontinuous (w.l.s.c.) at $Q$, we show that the two quantities are equal; in particular, this happens when $\Pcal$ is weakly compact. For simple alternatives, we provide the first matching necessary and sufficient condition for when power-one sequential tests exist (without assumptions on $\Pcal, Q$).
We also derive the optimal worst-case growth rate against composite $\Qcal$. 
We emphasize that test supermartingales on reduced filtrations suffice for all iid testing problems, and more general e-processes are not required. 
We thus completely generalize the recent results of Larsson et al.~\cite{larsson2025numeraire} to the sequential setting.
\end{abstract}


\section{Introduction}
Consider a forecaster who asserts that all observations in a data stream are drawn i.i.d.\ from some unknown distribution $P\in\Pcal$. In response, a skeptic, who believes the data are drawn i.i.d.\ from some $Q$, tries to falsify this claim by sequentially betting on unseen outcomes at odds that are fair to the forecaster. The skeptic's accumulated wealth represents the evidence gained against the null: the larger the wealth, the more the evidence. These wealth processes are nonnegative supermartingales, or, more generally, $e$-processes (defined later). These concepts have been very well-studied in the recent literature \parencite{10.1111/rssa.12647,ramdas2023,grunwald2024,ramdas2022testing}, and can be intuited as follows. If the null hypothesis is true (i.e.\ the forecaster is correct), the game has been set up such that the skeptic cannot expect to increase their wealth by betting against it. However, under the alternative $Q$, the skeptic's fortune can be made to consistently grow, exponentially fast.
Our main contribution is characterizing the optimal asymptotic growth rate of evidence under the alternative. 
This problem has a rich history, dating back to Kelly~\cite{kellyinterpretation}, Breiman~\cite{breiman1961optimal} and Cover~\cite{cover1987log}, who studied the problem in different settings. However, ours is the first work to completely characterize the asymptotically optimal exponent at which the wealth can grow, under no assumptions on $\Pcal$ and $Q$. (We also characterize the optimal worst-case growth rate when $Q$ is only known to lie in some arbitrary set $\Qcal$.) 

Now, against a point-alternative $Q$, the most common quantifier of the separation from the null $\Pcal$ is
\[
\KLinf(Q,\Pcal):=\inf_{P\in\Pcal}\KL(Q\|P).
\]
This $\KLinf$ has been presented as the correct benchmark in the literature, as pointed out in several works \cite{agrawal2025stoppingtimespoweronesequential, ram2026asymptoticallyoptimalsequentialchange, sethi2026asymptoticallyoptimalsequentialtesting, shekhar2025optimalanytimevalidtestscomposite}. Unfortunately, this will be only be the correct maximal growth rate under appropriate conditions on $\Pcal$ (like convexity and compactness). To this end, a growth rate of $\KLinf$ will not be achieved by many valid sequential procedures and this quantity does not govern the growth of sequential evidence\footnote{See Proposition~\ref{prop:counterexample} for an example where any $e$-process' growth is $\leq 0$, while  $\KLinf$ is positive (but $\infKL=0$).}. Instead of the $\KLinf$, the right quantity to analyze is the reverse information projection \parencite{larsson2025numeraire} onto the bipolar enlargement of $\Pcal$, a quantity that we call the $\infKL$. To define it, first let
\[
a_n(Q,\Pcal):=\inf_{R\in(\Pcal^n)^{\circ\circ}}\KL(Q^n\|R) =: \infKL(Q^n,\Pcal^n),
\]
where $\Pcal^n := \{P^n : P \in \Pcal \}$, and $(\Pcal^n)^{\circ\circ}$ is its \emph{bipolar} (defined later).
In Theorem~\ref{thm:max-eprocess-rate}, we show that the maximal asymptotic expected-log growth rate against a point alternative $Q$ equals the limit of  $a_n(Q,\Pcal)/n$. That is,
$$
 \sup_{\text{wealth processes } W} \limsup_{n \to \infty} \frac1n \E_Q[\log W_n]   = \lim_{n\to \infty} a_n(Q,\Pcal)/n.
$$ 
In other words, to understand the limit of accumulation of sequential evidence that any procedure can achieve, one would need to analyze the bipolar $(\Pcal^n)^{\circ\circ}$ rather than $(\Pcal^n)$.

To better intuit the importance of the bipolar, note that the polar $(\Pcal^n)^{\circ}$ is the class of all nonnegative $n$-sample tests ($e$-variables) that have expectation at most one under every $P\in\Pcal$. And, $(\Pcal^n)^{\circ\circ}$ is the corresponding dual class, which is termed by \cite{larsson2025numeraire} as the \emph{effective null hypothesis}. Hence intuitively, $a_n(Q,\Pcal)$ is the reverse $\KL$ distance from the alternative law $Q^n$ to the class that is indistinguishable from $\Pcal^n$ by $e$-variables. 

We show in Lemma~\ref{lem:no-power-one-if-in-bipolar} that if $Q^n\in(\Pcal^n)^{\circ\circ}$ for every $n$, then it is impossible to have any level-$\alpha$ power one sequential test against $Q$. In Proposition~\ref{prop:powerone-implies-out} we present a complementary result showing how power-one testing entails eventual non-membership in the bipolar. As such, in this paper we unify separation from the bipolar with attainable wealth growth and existence of sequential tests. Thus, to our knowledge, we are the first work to analyze the interplay between these, extending the non-sequential work of~\cite{larsson2025numeraire} to the sequential setting. In particular, we provide the first matching necessary and sufficient condition in the literature for a sequential power one test  to exist for any null $\Pcal$ against any alternative $Q$.

As mentioned above, the intrinsic growth rate we establish in this paper, the $\infKL$, will be smaller than the $\KLinf$. However, letting $\Phi(R):=\inf_{P\in\Pcal}\KL(R\|P)$, Theorem~\ref{thm:main} proves that if $\Phi$ is weakly lower semicontinuous at $Q$, then
\[
\lim_{n\to\infty} \infKL(Q^n,\Pcal^n) =\KLinf(Q,\Pcal).
\]
Fortunately, this condition is only local at the alternative (as opposed to something more global, like a dominating reference measure, an assumption we always avoid). We note that weak compactness of $\Pcal$ is sufficient for this rate identity to hold. And under this same condition, in Proposition~\ref{prop:eprocess-rate}, we build a test supermartingale \emph{in a reduced filtration} that, under $Q$, achieves almost sure and expected-log growth rate of $\KLinf(Q,\Pcal)$. And if we threshold this process, we obtain (in Theorem~\ref{thm:powerone-characterization}) power-one sequential tests whenever this limiting rate is positive and vice-versa. Intuitively, one can see that under weak lower semicontinuity, the typical $\KLinf(Q,\Pcal)$ once again gains meaning. 
Further,  Proposition~\ref{prop:infkl-klinf} shows that the assumptions of weak compactness and convexity of $\Pcal$ together\footnote{Unfortunately, weak compactness alone is not enough without convexity. See Proposition~\ref{prop:weakly-compact-not-enough} for an example.} imply that
\[
a_1(Q,\Pcal)= \infKL(Q,\Pcal)=\KLinf(Q,\Pcal).
\]
In addition to point alternatives, we present many results for composite alternatives also, where $Q$ is only know to lie in some set $\Qcal$, and one would like to maximize the worst case growth rate of wealth over $\Qcal$. 
A natural guess for the answer could be
\[
\overline d_{\mathrm{wc}}:=\limsup_{n\to\infty}\frac1n \inf_{Q\in\Qcal} a_n(Q).
\]
Unfortunately, this in general is not the tight characterization of the maximal growth rate. To present the correct answer, first define the so-called $\mathrm{GROW}$ value \parencite{grunwald2024, arnold2026optimalevaluestestingmean} after $n$ observations as:
\[
b_n(\Qcal,\Pcal):=\sup_{E\in(\Pcal^n)^{\circ}}\inf_{Q\in\Qcal}\E_{Q^n}[\log E].
\]
This payoff maximizes the worst-case expected log-evidence over \emph{non-sequential} $n$ sample tests. Theorem~\ref{thm:composite-maxrate} proves that the intrinsic maximal uniform expected-log growth rate of any $e$-process against $\Qcal$ is a sequential extension of $\mathrm{GROW}$, which we denote $\overline d_{\mathrm{rob}}$. That is, we prove that
\[
\overline d_{\mathrm{rob}}  = \sup_{\text{wealth processes } W} \inf_{Q \in \Qcal} \limsup_{n \to \infty} \frac1n \E_Q[\log W_n] = \limsup_{n\to\infty}\frac{b_n(\Qcal,\Pcal)}{n}.
\]
Remarkably, this characterization is under no assumptions on $\Pcal$ and $\Qcal$. 
In many settings, one can further rewrite this $\mathrm{GROW}$ in terms of the relative entropy between $\Pcal$ and $\Qcal$. For example, if $\Qcal$ is convex and compact in the setwise topology, and $\Pcal$ is convex and weakly compact, then $b_n(\Qcal,\Pcal)/n = \inf_{Q \in \Qcal}\inf_{P \in \Pcal}\KL(Q,P)$. We give more such examples later.  

In general, we show (in Lemma~\ref{lem:bn-le-an}) that $\overline d_{\mathrm{rob}}\le \overline d_{\mathrm{wc}}$. 
We give two cases where these two values coincide: for both finite $\Qcal$ and more generally under a sub-exponential covering condition, there will be no asymptotic mini-max gap. These conditions are presented in Proposition~\ref{prop:finite-Q-gap} and Theorem~\ref{thm:subexp-cover}.

Finally, in Theorem~\ref{thm:uniform-power-one-compact}, we show that if  $\inf_{Q\in\Qcal}\KLinf(Q,\Pcal)>0$ along with $\Pcal$ and $\Qcal$ both being weakly compact, we can build a single test supermartingale with strictly positive asymptotic growth rate for every $Q\in\Qcal$. And if we threshold it, we obtain a power-one sequential test. 

\paragraph{Related Work.}
A number of works have established martingale arguments for constructing tests that remain valid under optional stopping \parencite{ville1939etude, doob1953stochastic, robbins1970}. At the same time, Kelly's criteria and later improvements connected expected log-wealth growth to relative entropy in simple settings and settings with the existence of a dominating reference measure  \parencite{kellyinterpretation, cover1991universal, CoverThomas2006}. In contrast, our work shows what will replace the $\KL$ identity when dealing with composite nulls and without a dominating reference measure (all the while ensuring validity at all stopping times).

More recently, a number of works on $e$-values and $e$-processes reframe nonnegative supermartingales as anytime-valid measures of statistical evidence against the null \parencite{shafer2011, howard2021, VovkWang2021, ramdas2023, KoningMeer2026}. As a consequence, issues with testing and reproducibility have (at least partially) been resolved. That is, we have both safe continuously monitored testing and calibration and merging of evidence across different analyses. Further, in this setting, general-purpose constructions and optimality notions like admissibility have been developed \parencite{ramdas2020admissible, grunwald2024}. Our work is the first assumption-free converse of this: we provide an achievability theorem for i.i.d.\ composite null testing. In doing so, we identify the maximal per-sample log-growth rate possible. Furthermore, for i.i.d.\ problems, we prove that optimizing over test supermartingales suffice in the sense that they already attain the full $e$-process envelope.

More specifically, in this safe-testing literature, growth-rate optimal ($\mathrm{GRO}$) $e$-variables for composite hypotheses have been established, along with its strong duality characterization terms in terms of the reverse information projection \parencite{grunwald2024, HarremoesLardyGrunwald2023, larsson2025numeraire}. Our work here is the asymptotic analogue of these finite-horizon results. We show that the sequential limiting growth rate (for i.i.d.\ data) is a normalized reverse $\KL$ to the bipolar $(\Pcal)^{\circ\circ}$. We emphasize here the evident gap between this bipolar rate and the well-studied $\KLinf(Q,\Pcal)$. We emphasize that, in general, the correct maximal growth rate of the $\infKL$ is only equal to the $\KLinf$ under certain topological conditions, such as weak lower semicontinuity. 

For composite hypotheses without a dominating reference measure, \cite{larsson2026completecharacterizationtestablehypotheses}  have completed Le Cam's foundational program~\cite{lecam1986asymptotic} of when nontrivial fixed-sample tests exist; in particular, one must pass to the weak-$\star$ closures in the dual space of finitely additive charges. Separately, recent work has also presented weak compactness of the null class as a general sufficient condition that guarantees existence of a power-one sequential test against the entire complement of the composite null $\Pcal$ \parencite{ram2026powersequentialtestsexist}. For such tests, it has been established that the sharp general lower-bound for any sequential procedure in the low type I error regime (ie as $\alpha\downarrow 0$) is $\log(1/\alpha)/\KLinf$ \parencite{agrawal2025stoppingtimespoweronesequential}. This bound has been further analyzed in terms of what constructions asymptotically match this \parencite{shekhar2025optimalanytimevalidtestscomposite}. However, none of these works provide an assumption-free condition for existence of power-one tests that is both \textit{necessary and sufficient} for point-alternatives. We explicitly do this (along with sufficient conditions for composite alternatives) in terms of strict positivity of $a_n(Q)$. In addition, we clarify under what restrictions the very-well studied (yet unfortunately incorrect) $\KLinf$ term is the governing complexity measure for sequentially testing composite nulls. 

\paragraph{Paper outline.}
The rest of this paper is organized as follows. In Section~\ref{sec:warmup}, we present all the foundational lemmata to understand the main results. In Section~\ref{sec:point-alternative-infKL}, we present our assumption-free growth theorem for point alternatives, along with several consequences in power-one testing, including a necessary and sufficient condition for existence. In Section~\ref{sec:bipeqklinf}, we present sufficient conditions under which the $\KLinf$ is the limiting growth rate. In Section~\ref{sec:counterx}, we highlight numerous examples and counter-examples that show the consequences of what happens if our hypotheses in these theorems are violated. In particular, perhaps the most important we show is why weak lower semicontinuity is only sufficient and not necessary. In Section~\ref{sec:compQintrins}, we illustrate the maximal intrinsic growth rate in the composite alternative setting, as well as some consequences when this intrinsic growth rate is positive. In this section, we also present results characterizing existence of power-one sequential tests under the setting of composite null versus composite alternatives. We show that test supermartingales suffice for achievability in Section~\ref{sec:super} and conclude and present open problems in Section~\ref{sec:cute}. All omitted proofs can be found in Section~\ref{sec:proofs-apendix}.

\section{Preliminaries}\label{sec:warmup}
Let $(\X, \B)$ be a measurable space.
We let $\mathcal{M}_1$ be the set of probability measures on $(\X,\B)$. (If needed, we may assume that $(\X, \B)$ is a Polish space with its Borel $\sigma$-field, $\B$.)

\begin{definition}\label{def:bipolar}
For any measurable space $(\Omega, \mathcal{F})$ and nonempty set of probability measures $\mathcal{S}$ on this space, 
define its polar as
\[
\mathcal{S}^{\circ}:=\Bigl\{E:\Omega\to[0,\infty]:\sup_{P\in\mathcal{S}} \E_P[E] \le 1\Bigr\}.
\]
Following~\cite{larsson2025numeraire}, the bipolar is then defined as
\[
\mathcal{S}^{\circ\circ} := \Bigl\{R\in\mathcal{M}_1(\Omega): \E_R[E] \le 1\ \text{ for all }E\in\mathcal{S}^{\circ} \Bigr\}.
\]
\end{definition}

\noindent Note that $\mathcal S\subseteq \mathcal S^{\circ\circ}$, because if $P\in\mathcal P$ and $E\in\mathcal S^{\circ}$, then $\E_P[E]\le \sup_{P'\in\mathcal S}\E_{P'}[E]\le 1$. 
We use $\Omega$ above instead of $\B$ because, for example, we will also use it on product spaces like $(\X^n, \B^{\otimes n})$. For $P \in \mathcal{M}_1(\X)$, let $P^n$ be the $n$ fold product measure on $(\X^n, \B^{\otimes n})$.

For $\Pcal \subseteq \mathcal{M}_1(\X)$ and $Q \in \mathcal{M}_1(\X)$, 
recall  $\KLinf(Q,\Pcal):=\inf_{P\in\Pcal}\KL(Q\|P)$ and 
define
\[
\infKL(Q^n,\Pcal^n):=\inf_{R\in(\Pcal^n)^{\circ\circ}} \KL(Q^n\|R)=\KLinf(Q^n,(\Pcal^n)^{\circ\circ})\in[0,\infty].
\]
It is shown that the above infimum is attained by a subprobability measure $P_n^\star$ (unique up to $Q$ null sets) called the \emph{reverse information projection} (RIPr) of $Q^n$ onto $(\Pcal^n)^{\circ\circ}$ \parencite{larsson2025numeraire}. 
One can show that for $n \geq 1$:
\[
\frac{1}{n}\,\infKL(Q^n,\Pcal^n) \leq \KL_{\inf}(Q,\Pcal).
\]
We will later show that in general the above inequality can be strict, but that equality holds under some simple structural assumptions on $\Pcal$, even without assuming a dominating reference measure. 

\begin{definition}\label{def:kl}
For $M, N\in\mathcal{M}_1(\X)$ we denote the KL divergence as
\[
\KL(M\|N):=
\begin{cases}
\displaystyle \int_\X \log\left(\frac{dM}{dN}\right)dM, &\text{if } M\ll N,\\[0.6em]
+\infty, &\text{otherwise.}   
\end{cases}
\]
For $p,q\in[0,1]$, the binary KL divergence is defined as
\[
D(p\|q):=p\log\frac{p}{q} + (1-p)\log\frac{(1-p)}{(1-q)},
\]
with the following conventions: if $q\in[0,1]$, $0\log (0/q):=0$, 
and if $p>0$, we take $p\log(p/0):=+\infty$, while $\log 0:=-\infty$. If $Z:\Omega\to[0,\infty]$ is measurable and $\mu\in\mathcal M_1(\Omega)$, then
\[
\E_{\mu}[\log Z]:=\sup_{M\ge 1}\int_{\Omega}\log(Z\wedge M)\,d\mu\in[-\infty,\infty].
\]
\end{definition}

\noindent 
Weak convergence of measures is denoted by $\Rightarrow$.  
The map $(M, N)\mapsto\KL(M\|N)$ from $\mathcal{M}_1(\X)\times\mathcal{M}_1(\X)$ to $[0,\infty]$ is lower semicontinuous for the weak topology on both coordinates.

To further clarify,  $\mathcal M_1(\X)$ is endowed with the weak topology, which is the coarsest topology for which the $P\mapsto \int f dP$ is continuous for every bounded continuous $f\in C_b(\X)$. Equivalently, $P_n\Rightarrow P$ if and only if for every $f\in C_b(\X)$,
\(
\int f\,dP_n\to\int f\,dP.
\)
We say that a subset $\Pcal\subset \mathcal M_1(\X)$ is weakly compact if it is compact in this topology. Because $\X$ is Polish, the weak topology on $\mathcal M_1(\X)$ is metrizable; for example, by the bounded Lipschitz metric. From this, one can interpret weak compactness sequentially. That is, every sequence $(P_n)\subseteq\Pcal$ has a subsequence $P_{n_k}\Rightarrow P$ for some $P\in\Pcal$. Prokhorov's theorem implies that $\Pcal$ being weakly compact is equivalent to $\Pcal$ being weakly closed and tight (for every $\varepsilon>0$, there exists some compact set $K_\varepsilon\subseteq \X$ such that $\sup_{P\in\Pcal} P(K_{\varepsilon}^c)\le \varepsilon$). Examples of weakly compact $\Pcal$ include all laws supported on a compact set in $\mathbb R^d$ (eg:  distributions on $[0,1]$ with mean  $\leq 0.5$), or exponential families where the parameter is in a closed set in $\mathbb R^d$, etc. 

Our proof of asymptotic exactness will need a way to choose a weak neighborhood of $Q$ where the composite divergence $\Phi(R):=\inf_{P\in\mathcal P}\KL(R\|P)$ does not drop substantially below $\Phi(Q)=\KLinf(Q,\Pcal)$. Note that $\Phi$ can fail to be lower semicontinuous if we lack some compactness or attainment property on $\Pcal$. 
We say that $\Pcal$ is \underline{$\KLinf$ lower semicontinuous at $Q$} if $\Phi$ is lower semicontinuous at $Q$ for the weak topology. Meaning that if for every sequence $R_k\Rightarrow Q$, we have also $\Phi(Q)\le\liminf_{k\to\infty}\Phi(R_k)$.   
As an important sufficient condition, if $\Pcal$ is weakly compact in $\mathcal{M}_1(\X)$, then $\Pcal$ is $\KLinf$ lower semicontinuous at every $Q$ \cite{ram2026powersequentialtestsexist}. 

We will now clarify some important conventions and the filtrations we are working on. Unless otherwise stated, all the sequential processes in this work are defined on $\Omega=\X^{\mathbb N}$ with coordinate maps $X_1, X_2, \dots$, and we set $\mathcal F_0:=\{\varnothing, \Omega\}$ and $\mathcal F_n:=\sigma(X_1,\dots,X_n)$ for $n\ge 1$. We write $\bbF:=(\mathcal F_n)_{n\ge 0}$ for the original filtration. We define (strictly) increasing sequences of time indices as 
\[
\Tdet:=\Big\{\bt=(t_k)_{k\ge 0}:= 0=t_0<t_1<t_2<\dots,\, t_j\in\mathbb N \text{ for all $j \geq 1$}
\Bigr\}.
\]
For $\bt\in\Tdet$, we define the reduced blockwise filtration as $\bbF_{\bt}:=(\mathcal F_{t_k})_{k\ge 0}$. Therefore, a $\bbF_{\bt}$ stopping time is a map $\tau:\Omega\to\{t_k:k\ge 0\}$ such that $\{\tau\le t_k\}\in \mathcal F_{t_k}$ for every $k\ge 0$. Note that when $\bt = \mathbb N_0$, we have $\bbF_{\bt} = \bbF$, so the blockwise filtrations recover the original filtration as a special case. 

More generally, if $\bbG=(\mathcal G_n)_{n\ge 0}$ is a filtration on the same space we write $\bbG\preceq \bbF$ whenever $\mathcal G_n\subseteq \mathcal F_n$ for all $n\ge 0$. For $\bt\in\Tdet$, we write $\bbG_{\bt}:=(\mathcal G_{t_k})_{k\ge 0}$ as the corresponding reduced blockwise filtration.

\begin{definition}[E-processes (on possibly  reduced filtrations)]\label{def:eprocess}
We say that a nonnegative adapted process $E=(E_s)_{s\in \bt}$ is an \underline{$e$-process for $\Pcal$ on $\bbG_{\bt}$} if for every $P\in\Pcal$ and every finite-valued $\bbG_\bt$-stopping time $\tau$, we have $\E_{P^{\infty}}[E_{\tau}]\le 1$. The class of all such processes is denoted by $\mathcal E_{\bbG_\bt}(\Pcal)$. 
The elements of $\mathcal E(\Pcal):= \mathcal E_{\bbF}(\Pcal)$ 
are simply called \underline{$e$-processes}, while those in some generic $\mathcal E_{\bbF_\bt}(\Pcal)$ are called \underline{blockwise e-processes}. A blockwise e-process $E$ (on some $\bbF_\bt$) can be lifted to an e-process (on $\bbF$) by defining $E_j = 0$ for $j \notin \bt$. 
\end{definition}

An e-process corresponds to the wealth process of a gambler who bets on every observation (technically, they can bet separately against each $P \in \Pcal$, but their net wealth is the infimum of their wealths in each game, see~\cite{ramdas2022testing} for details).
A blockwise e-process arises because a gambler does not wish to bet on every observation one by one, but bets on blocks of observations instead. An important class of (blockwise) e-processes are (blockwise) test supermartingales, defined next.

\begin{definition}[Blockwise test supermartingale]\label{def:test-supermartingale}
Let $\bt=(t_k)_{k\ge 0}\in \Tdet$. A blockwise test supermartingale for $\Pcal$ on $\bbF_{\bt}$ is a nonnegative process $(S_{t_k})_{k\ge 0}$ such that $S_{t_0}=1$, each $S_{t_k}$ is $\mathcal F_{t_k}$ measurable and for every $P\in\Pcal$ we have that for all $k\ge 1$ almost surely under $P^{\infty}$,
\[
\E_{P^{\infty}}[S_{t_k}\mid \mathcal F_{t_{k-1}}]\le S_{t_{k-1}}.
\]
\end{definition}
Due to the optional stopping theorem, every blockwise test supermartingale on $\bbF_{\bt}$ is a blockwise $e$-process, hence an element of $\mathcal E_{\bbF_{\bt}}(\Pcal)$. 

\begin{definition}[Wealth processes]
For any subfiltration $\bbG\preceq\bbF$ we define the class of wealth processes on $\bbG$ as
\[
\Wcal_{\bbG}(\Pcal):=
\bigcup_{\bt\in\Tdet}\mathcal E_{\bbG_{\bt}}(\Pcal).
\]
Therefore, a wealth process is either an $e$-process on $\bbG$ or an $e$-process on some reduced block filtration $\bbG_{\bt}$. In the special case of $\bbG=\bbF$, we write $\Wcal(\Pcal):=\Wcal_{\bbF}(\Pcal)$. Each $W\in\Wcal_{\bbG}(\Pcal)$ is understood together with its indexing filtration.
\end{definition}
Wealth processes are named as such because they correspond to the sequences of wealths that any gambler can possibly achieve by defining various betting games in which they can bet against the null $\Pcal$ (or more accurately, against $\Pcal^\infty$).


Throughout this work, all our statements about processes on the original filtration will refer to an $e$-process on $\bbF$. All statements about reduced filtrations will refer to a blockwise $e$-process, i.e.\ an $e$-process on some $\bbF_{\bt}$.


\section{Point $Q$: Intrinsic Growth Rate and Sequential Testability}\label{sec:point-alternative-infKL}

We will now explicate something seemingly contradictory to the literature. Often, the $\KLinf$ is seen as the key quantity in optimal bounds. However, we will show that in general, $e$-processes cannot grow any faster than the $\infKL$, and the $\KLinf$ may not ever be achieved. 

Recall the shorthands of $\Pcal^n:=\{P^n:P\in\Pcal\}$ and $a_n:=\infKL(Q^n,\Pcal^n):=\inf_{R\in(\Pcal^n)^{\circ\circ}}\KL(Q^n\|R)\in[0,\infty]$, where $(\Pcal^n)^{\circ\circ}$ is the bipolar. 
For a fixed horizon $n$, define the $e$-power, i.e.\ the expected log capital as
\[
\EPow(Q,\Pcal):=\sup_{E\in(\Pcal^n)^\circ}\E_{Q^n}[\log E].
\]
The main result of \cite{larsson2025numeraire} is the following strong duality, with no restrictions on $Q,\Pcal$.
\begin{proposition}[\cite{larsson2025numeraire}]\label{prop:infkl-e-power}
$\EPow(Q,\Pcal) = \inf_{R\in(\Pcal^n)^{\circ\circ}}\KL(Q^n \|R)$.
\end{proposition}

We will now explicate a very useful observation about these $n$-step bipolars.

\begin{lemma}\label{lem:an-superadditive-sec}
For each $Q\in\mathcal M_1(\X)$, the sequence $(a_n(Q))_{n\ge 1}$ is superadditive. Meaning for all $n,m\in\mathbb N$, $a_{n+m}(Q)\ge a_n(Q)+a_m(Q)$. So it follows that by Fekete's lemma,
\[
d_{\star}(Q,\Pcal)=\lim_{n\to\infty}\frac1n a_n(Q)=\sup_{n\ge 1}\frac1n a_n(Q)\in[0,\infty].
\]
\end{lemma}

We will now present an impossibility result for $e$-processes on any filtration $\bbF_{\bt}$.

\begin{lemma}\label{lem:eprocess-upper-by-infkl}
Let $\bt=(t_k)_{k\ge 0}\in\Tdet$ and let $W=(W_{t_k})_{k\ge 0}\in\mathcal E_{\bbF_{\bt}}(\Pcal)$ be a blockwise $e$-process for $\Pcal$. Now, for each $n\in\mathbb N$, recall that $a_n:=\infKL(Q^n,\Pcal^n)=\inf_{R\in(\Pcal^n)^{\circ\circ}}\KL(Q^n\|R)\in[0,\infty]$. Then it follows that for every $k\ge 1$ we have that
\[
\E_{Q^\infty}[\log W_{t_k}] =\E_{Q^{t_k}}[\log W_{t_k}]\le a_{t_k}.
\]
And so consequently we have that
\[
\limsup_{k\to\infty}\frac{1}{t_k}\E_{Q^\infty}[\log W_{t_k}]\le \lim_{n\to\infty}\frac{1}{n}a_n.
\]
\end{lemma}

At this point, we are ready to present our achievability proposition, proving that an $e$-process can attain the $\infKL$ rate arbitrarily closely. In particular, let us denote the asymptotic $\infKL$ rate by
\[
\overline d_{\infKL}:=\limsup_{n\to\infty}\frac{1}{n}\infKL(Q^n,\Pcal^n)\in[0,\infty].
\]
And, if this limit does indeed exist, we will denote it by $d_{\infKL}$. In words: what our next proposition tells us that $\overline d_{\infKL}$ is not at all an upper bound only, rather it is achievable arbitrarily closely by a blockwise $e$-process. 

\begin{proposition}\label{prop:achieve-infkl-rate}
Assume that $\overline d_{\infKL}<\infty$. Then, for every $\varepsilon>0$, there exist $\bt=(t_k)_{k\ge 0}\in\Tdet$ and a blockwise test supermartingale $(E_{t_k})_{k\ge 0}$ for $\Pcal$ on $\bbF_{\bt}$ such that under $Q^{\infty}$,
\[
\boxed{
\lim_{k\to\infty}\frac{1}{t_k}\E_{Q^\infty}[\log E_{t_k}]\ge \overline d_{\infKL}-\varepsilon.
}
\]
And moreover,
\[
\lim_{k\to\infty}\frac{1}{t_k}\log E_{t_k} = \lim_{k\to\infty}\frac{1}{t_k}\E_{Q^\infty}[\log E_{t_k}] \quad Q^\infty\text{-a.s.}
\]
Thus, suppose $d_{\infKL}:=\lim_{n\to\infty}\frac1n\infKL(Q^n,\Pcal^n)$ exists and is finite. Then, for every $\varepsilon>0$ there exists a blockwise test supermartingale on some $\bbF_{\bt}$ with asymptotic growth rate at least $d_{\infKL}-\varepsilon$ both almost surely and in expectation.
\end{proposition}

\begin{proof}
Take some $\varepsilon>0$ and assume that $\overline d_{\infKL}<\infty$. Just for notational convenience, recall we let $a_m:=\infKL(Q^m, \Pcal^m)$. Necessarily by definition of $\limsup$ and the fact that $d_{\infKL}<\infty$, there exists some integer block length $m\in\mathbb N$ such that $a_m<\infty$ and that
\begin{equation}\label{eq:choose-m-close}
\frac{1}{m}\infKL(Q^m, \Pcal^m)\ge \overline d_{\infKL}-\frac{\varepsilon}{2}.
\end{equation}

The first thing we must do is choose a single $m$-sized block $e$-variable with a near optimal e-power. By Proposition~\ref{prop:infkl-e-power} at horizon $m$, $a_m=\sup_{E\in(\Pcal^m)^\circ}\E_{Q^m}[\log E]$. Therefore, there exists some $E^{\star}:\X^m\to(0,\infty)$ with $E^{\star}\in(\Pcal^m)^\circ$ such that
\begin{equation}\label{eq:E-star-near-opt}
\E_{Q^m}[\log E^\star]\ge a_m-\frac{\varepsilon m}{2}.    
\end{equation}

Note here that we are just using the fact that a supremum can be approximated arbitrarily closely. We now need to build an $e$-process by repeating this same $m$-block e variable. To this end, we will partition the coordinates into consecutive blocks of length $m$. Then, block $k$ consists of $(X_{(k-1)m+1},\dots,X_{km})$. And, define $t_k:=km$, then $\mathcal F_{t_k}:=\sigma(X_1,\dots,X_{t_k})$ is $\bbF_{\bt}$. Now define $E_{t_0}:=1$ and for $k\ge 1$,
\[
E_{t_k}:=\prod_{j=1}^k E^\star\bigl(X_{(j-1)m+1},\dots,X_{jm}\bigr).
\]
We claim this satisfies the definition of blockwise $e$-process for any $P\in\Pcal$. Because of the fact that $P^\infty=P^{\otimes\mathbb N}$ is iid, it must be the case that the $k$-th block is independent of $\mathcal F_{t_{k-1}}$ and distributed as $P^m$. Hence for all $k\ge 1$ we have that
\begin{align*}
\E_{P^\infty}[E_{t_k}\mid\mathcal F_{t_{k-1}}]&=\E_{P^\infty}\left[ E_{t_{k-1}}\cdot E^\star\bigl(X_{(k-1)m+1},\dots,X_{km}\bigr)\middle|\mathcal F_{t_{k-1}}\right]\\
&=E_{t_{k-1}}\cdot \E_{P^m}[E^\star].
\end{align*}

Now, note that $E^\star\in(\Pcal^m)^\circ$. As such, it is clear that we have by definition that $\sup_{P\in\mathcal P}\E_{P^m}[E^\star]\le 1$. So in particular, $\E_{P^m}[E^\star]\le 1$. Therefore we get that
\[
\E_{P^\infty}[E_{t_k}\mid\mathcal F_{t_{k-1}}]\le E_{t_{k-1}}.
\]
Therefore, clearly, $(E_{t_k})_{k\ge 0}$ is a test supermartingale for $\Pcal$, i.e.\ it satisfies Definition~\ref{def:test-supermartingale}. Now, we're going to compute the expected log growth rate under $Q^\infty$. Under $Q^\infty=Q^{\otimes\mathbb N}$, we know that the blocks are iid with law $Q^m$. Given this, for $j\ge 1$ define the iid random variables
\[
Y_j:=\log E^\star\bigl(X_{(j-1)m+1},\dots,X_{jm}\bigr).
\]
So, $\log E_{t_k}=\sum_{j=1}^k Y_j$ and $\E_{Q^\infty}[\log E_{t_k}]=\sum_{j=1}^k \E_{Q^\infty}[Y_j] = k\E_{Q^m}[\log E^\star]$. If we divide by $t_k=km$ we get that
\begin{equation}\label{eq:expected-rate}
\frac{1}{t_k}\E_{Q^\infty}[\log E_{t_k}]=\frac{1}{m}\E_{Q^m}[\log E^\star].   
\end{equation}

Now, then applying \eqref{eq:E-star-near-opt} followed by \eqref{eq:choose-m-close} we get that
\[
\frac{1}{m}\E_{Q^m}[\log E^\star] \ge \frac{a_m}{m}-\frac{\varepsilon}{2} \ge \overline d_{\infKL}-\varepsilon.
\]
Taken in conjunction with \eqref{eq:expected-rate}, this proves
\[
\lim_{k\to\infty}\frac{1}{t_k}\E_{Q^\infty}[\log E_{t_k}]\ge \overline d_{\infKL}-\varepsilon.
\]

It remains to now prove the almost sure convergence of the realized log growth rate. Because $a_m<\infty$ and $E^\star\in(\Pcal^m)^\circ$, Proposition~\ref{prop:infkl-e-power} entails $\E_{Q^m}[\log E^\star]\le a_m<\infty$. Therefore, taking this with \eqref{eq:E-star-near-opt}, $\E_{Q^m}[\log E^\star]\in\mathbb R$. So, $Y_1=\log E^\star(X_{1:m})$ is integrable. Meaning, $\E_{Q^m}[|Y_1|]<\infty$. By the strong law of large numbers applied to the iid sequence $(Y_j)$ we get that almost surely under $Q^\infty$,
\[
\frac{1}{k}\sum_{j=1}^k Y_j\to\E_{Q^m}[Y_1]=\E_{Q^m}[\log E^\star].
\]

Recall that $t_k=km$. Dividing by $m$ proves that almost surely under $Q^\infty$,
\[
\frac{1}{t_k}\log E_{t_k} \to \frac{1}{m}\E_{Q^m}[\log E^\star].
\]
Therefore, alongside \eqref{eq:expected-rate}, we have shown that the almost sure limit indeed equals the expected log limit, proving the claim.
\end{proof}

We can summarize the intrinsic maximal expected-log growth with the following theorem. Since each $W\in\Wcal(\Pcal)$ actually lies in $\mathcal E_{\bbF_{\bt}}(\Pcal)\text{ for some }\bt\in\Tdet$, we define its point-alternative expected log-growth functional as
\[
\mathcal R_Q(W):= \limsup_{k\to\infty}\frac{1}{t_k}\E_{Q^\infty}[\log W_{t_k}],  \text{ where } \bt\in\Tdet \text{ is such that } W\in\mathcal E_{\bbF_{\bt}}(\Pcal).
\]
Recall that here, $\bt$ is the unique block schedule that is associated with $W$.

\begin{theorem}\label{thm:max-eprocess-rate}
Let $\overline d_{\infKL}:=\limsup_{n\to\infty}\frac{1}{n}\infKL(Q^n,\Pcal^n) =\limsup_{n\to\infty}\frac{1}{n}a_n\in[0,\infty]$. Then it follows that
\[
\boxed{
\sup_{W\in\Wcal(\Pcal)}\mathcal R_Q(W)=\ \overline d_{\infKL}.
}
\]
\end{theorem}

\begin{proof}
We will start with the upper bound. Suppose that $W\in\Wcal(\Pcal)$. If $W\in\mathcal E_{\bbF}(\Pcal)$, we can apply Lemma~\ref{lem:eprocess-upper-by-infkl} with $t_k=k$. If $W\in\mathcal E_{\bbF_{\bt}}(\Pcal)$ for some $\bt=(t_k)_{k\ge 0}$, we can apply Lemma~\ref{lem:eprocess-upper-by-infkl} using this $\bt$. In either case it follows that
\[
R_Q(W)\le \limsup_{n\to\infty}\frac1n a_n = \overline d_{\infKL}.
\]
If we take the supremum over all $W\in\Wcal(\Pcal)$, that proves the upper bound. We will now prove the \underline{lower bound,} i.e.\ the achievability. We are first going to handle the case where $\overline d_{\infKL}=+\infty$. To this end, let $M>0$ be arbitrarily chosen. By definition of $\limsup$, there exists $m\in\mathbb N$ such that $a_m=\infKL(Q^m, \Pcal^m)\ge mM+1$. Note that this inequality is understood to be in $[0,\infty]$. We emphasize this because it allows the bipolar $a_m=+\infty$ in particular. Proposition~\ref{prop:infkl-e-power} yields $a_m=\sup_{E\in(\Pcal^m)^\circ}\E_{Q^m}[\log E]$. So there exists an $e$-variable $E^\star\in(\Pcal^m)^\circ$ such that $\E_{Q^m}[\log E^\star]\ge mM$. Let us now partition the data into consecutive blocks of length $m$ and set $t_k:=km$. Then, define
\[
W_{t_k}:=\prod_{j=1}^k E^\star\bigl(X_{(j-1)m+1},\dots,X_{jm}\bigr).
\]

We have that $(W_{t_k})_{k\ge 0}$ is a test blockwise test supermartingale along $\bbF_{\bt}$. Therefore, $W\in\Wcal(\Pcal)$. Thus under $Q^\infty$ by independence of the blocks we get
\[
\frac{1}{t_k}\E_{Q^\infty}[\log W_{t_k}] =\frac{1}{m}\E_{Q^m}[\log E^\star] \ge M.
\]
Notice how $M>0$ was arbitrary. Hence, the supremum of achievable expected log growth rates is indeed $+\infty$. Now, in the case where $\overline d_{\infKL}<\infty$, Proposition~\ref{prop:achieve-infkl-rate} already gives us for every $\varepsilon>0$, a blockwise test supermartingale $W\in\Wcal(\Pcal)$ such that
\[
R_Q(W)=\lim_{k\to\infty}\frac{1}{t_k}\E_{Q^\infty}[\log W_{t_k}] \ge \overline d_{\infKL}-\varepsilon.
\]
Then, letting $\varepsilon\downarrow0$ proves the lower bound, completing the proof.
\end{proof}

Therefore, Theorem~\ref{thm:max-eprocess-rate} tells us that unfortunately the $\KLinf(Q,\Pcal)$ isn't the intrinsic maximal expected log growth rate of $e$-processes in general unfortunately. The right maximal rate is in fact the asymptotic $\infKL$ rate $\overline d_{\infKL}=\limsup_n \frac1n\infKL(Q^n,\Pcal^n)$. Meaning, $\overline d_{\infKL}$ is always achievable arbitrarily closely and it upper bounds all $e$-processes. 
As we will see in our Proposition~\ref{prop:counterexample}, we have that $\infKL(Q^n, \Pcal^n)=0$ for all $n$, so $\overline d_{\infKL}=0$. Hence the above upper bound that we showed implies that no $e$-process (in the most general sense) can have positive asymptotic expected log growth under $Q^\infty$. And here we see the constant process $E_{t_k}\equiv 1$ achieves rate $0$. So we can have cases where the $\KLinf(Q,\Pcal)>0$ while the maximum achievable $e$-process rate is $0$.


\begin{remark}\label{rem:blockwise-versus-full-point-alternative}
Now, note that the supremum over wealth processes will be equal to the supremum over $e$-processes on the original filtration $\bbF$. In particular, more generally, our zeroing argument will work for any subfiltration $\bbG\preceq\bbF$. Suppose that $W=(W_{t_k})_{k\ge 0}\in\mathcal E_{\bbG_{\bt}}(\Pcal)$ is a blockwise $e$-process for some $\bt=(t_k)_{k\ge 0}\in\Tdet$. Define the zeroed process $\widehat W=(\widehat W_n)_{n\ge 0}$ by
\[
\widehat W_n:=
\begin{cases}
W_{t_k},&n=t_k\,\text{ for some $k\ge 0$},\\
0,&n\notin\{t_k:k\ge 0\}.
\end{cases}
\]
It follows then that $\widehat W\in\mathcal E_{\bbG}(\Pcal)$. Indeed, suppose that $\tau$ is any finite valued $\bbG$-stopping time. Then, the ceiling time $\tau^+:=\min\{t_k:t_k\ge \tau\}$ will again be a $\bbG_{\bt}$-stopping rule. So, we will have $\widehat W_{\tau}\le W_{\tau^+}$ pointwise. Therefore, for every $P\in\Pcal$,
\[
\E_{P^\infty}[\widehat W_\tau]\le \E_{P^\infty}[W_{\tau^+}]\le 1.
\]
Thus, through zeroing, we can embed every blockwise $e$-process on $\bbG_{\bt}$ into an $e$-process on $\bbG$. In particular, with $\bbG=\bbF$ and our convention that takes $\log 0 = -\infty$,
\[
\limsup_{n\to\infty}\frac{1}{n}\E_{Q^\infty}[\log \widehat W_n] = \limsup_{k\to\infty}\frac{1}{t_k}\E_{Q^\infty}[\log W_{t_k}].
\]
Now since every $e$-process on $\bbG$ is already a wealth process on $\bbG$, this proves that
\[
\sup_{W\in\Wcal_{\bbG}(\Pcal)}\mathcal R_Q(W) = \sup_{E\in\mathcal E_{\bbG}(\Pcal)} \limsup_{n\to\infty}\frac{1}{n}\E_{Q^\infty}[\log E_n].
\]
In particular, for the original filtration we have
\[
\boxed{
\sup_{W\in\Wcal(\Pcal)}\mathcal R_Q(W) = \sup_{E\in\mathcal E(\Pcal)} \limsup_{n\to\infty}\frac{1}{n}\E_{Q^\infty}[\log E_n].
}
\]
\end{remark}

One other idea would be to extend the blockwise $e$-process with the left-constant interpolation of $\widehat W_n:= W_{t_k}$ whenever $t_k\le n<t_{k+1}$. Unfortunately, this will fail to be an $e$-process on $\bbF$. The issue is that a $\bbF$-stopping time can use additional information that is revealed inside the current block and then decide to stop immediately or (strategically) wait until the next block time. We will show a situation where this happens in the following example.

\begin{example}\label{ex:counterzero-versus-constant-eproc}
Let $\X=\{0,1\}$. Let $P\sim\mathrm{Bern}(1/2)$. Let the block times be $t_0=0, t_1=2, t_2=4$ respectively. Now define $W_0:=1$, $W_2:=1$, and $W_4:=2\mathbf{1}_{\{\X_{3}=1\}}$. We claim that $(W_{t_k})$ is an $e$-process along $\{0,2,4\}$: a block-wise $e$-process. To see this, suppose $\sigma$ is any stopping rule taking values in $\{0,2,4\}$. Then, $\{\sigma=4\}=\{\sigma>2\}\in\mathcal F_2$. Now, $W_4$ depends only on $\X_3$, which is independent of $\mathcal F_2$. Further, $\E_P[W_4]=1$. Therefore, we get
\begin{align*}
\E_P[W_{\sigma}]&=\E_P[\mathbf{1}_{\{\sigma=0\}}W_0+\mathbf{1}_{\{\sigma=2\}}W_2+\mathbf{1}_{\{\sigma=4\}}W_4]\\
&= P(\sigma=0) + P(\sigma=2) + \E_P[\mathbf{1}_{\{\sigma=4\}}W_4]\\
&= P(\sigma=0)+P(\sigma=2)+P(\sigma=4)\E_P[W_4]\\
&= 1.
\end{align*}
Therefore, it follows that the block-wise $e$-property holds. In particular, $(W_{t_k})$ is an $e$-process on the reduced filtration. Let us now form the left-constant interpolation $\widehat W_n:=W_{t_k}$ for $t_k\le n<t_{k+1}$. Then, $\widehat W_3=W_2=1$ and $\widehat W_4=W_4=2\mathbf{1}_{\{X_3=1\}}$. Now consider the $\bbF$-stopping time
\[
\tau:=
\begin{cases}
3,&\X_3=0,\\
4,&\X_3=1.
\end{cases}
\]
This will be an $(\mathcal F_n)$-stopping rule. In addition,
\[
\widehat W_{\tau}=\mathbf{1}_{\{\X_3=0\}}\widehat W_3 + \mathbf{1}_{\{\X_3=1\}}\widehat W_4 = \mathbf{1}_{\{\X_3=0\}} + 2\mathbf{1}_{\{\X_3=1\}}.
\]
Therefore,
\[
\E_P[\widehat W_{\tau}]=\frac12 \cdot 1 + \frac12 \cdot 2 = \frac32 > 1.
\]
It follows that $\widehat W$ cannot be a $e$-process on the original filtration $\bbF$.
\end{example}
In Example~\ref{ex:counterzero-versus-constant-eproc}, we are viewing what occurs if one ``peeks'' into the next block. Upon seeing that $X_3=1$, one can wait until time $4$ to capitalize on the larger $W_4$. However, if $X_3=0$, one can stop immediately at time $3$, keeping the previous capital $W_2$. The blockwise $e$-process definition does not allow such peeking into the block and then choosing. This is why our embedding of the $e$-processes on $\bbF$ is doing so with zero off the block times, instead of holding it constant between the times. 

\subsection{A necessary condition for sequential testability}
Suppose that $\Pcal$ is a class of probability measures on some measurable space and that $Q$ is any probability measure. In particular, we will allow $Q\in\Pcal$, which will correspond to $\KLinf(Q,\Pcal)=0$. 
For each $n$, we let $\Pcal^n:\{P^n:P\in\Pcal\}$ will be the corresponding i.i.d.\ product class on $n$ samples, and $(\Pcal^n)^{\circ\circ}$ is its bipolar. We recall a lemma that gives intuition on the event probabilities over the bipolar, proved in \parencite{ram2026powersequentialtestsexist}. 

\begin{lemma}\label{lem:event-sup}
Suppose that $\mathcal{S}$ is a nonempty set of probability measures on $(\Omega, \mathcal{F})$ and let $A\in\mathcal{F}$. Then it follows that
\[
\sup_{R\in\mathcal{S}^{\circ\circ}} R(A)=\sup_{P\in\mathcal S}P(A).
\]
\end{lemma}

Now with this defined, we will now present a lemma that tells us that if $Q^n\in(\Pcal^n)^{\circ\circ}$ for all $n$, then power one tests are in fact impossible.

\begin{lemma}\label{lem:no-power-one-if-in-bipolar}
Assume that for every $n\in\mathbb N$ we have that $Q^n\in(\Pcal^n)^{\circ\circ}$ on $(\X^n,\B^{\otimes n})$. Then for every $\alpha\in(0,1)$ and every stopping time $\tau$, $\sup_{P\in\Pcal}P^\infty(\tau<\infty)\le\alpha$ implies that $Q^\infty(\tau<\infty)\le \alpha$. 
Said differently, if a power-one test exists for $Q$ against $\Pcal$, then it must be the case that for some $k \in \mathbb N$, $Q^k\notin(\Pcal^k)^{\circ\circ}$.
\end{lemma}

\begin{proof}
By definition for each $n$, the event $\{\tau\le n\}\in\mathcal F_n$ depends only on the first $n$ samples. Hence $Q^\infty(\tau\le n)=Q^n(\tau\le n)$. Now since $Q^n\in(\Pcal^n)^{\circ\circ}$, Lemma~\ref{lem:event-sup} gives
\[
Q^n(\tau\le n)\le \sup_{P\in\Pcal}P^n(\tau\le n)\le \sup_{P\in\Pcal}P^\infty(\tau<\infty)\le \alpha.
\]
Now taking the $\sup_n$ gives us by continuity of measure that $Q^\infty(\tau<\infty)=\sup_n Q^\infty(\tau\le n)\le \alpha$.
\end{proof}

Interestingly, non-membership in the bipolar will lead to a $\KL$ gap that is strictly positive. We will formalize this idea in Proposition~\ref{prop:notinbipolar-positive}. That is, we will show that for every nonempty $\mathcal S$ and $Q\notin S^{\circ\circ}$, $\inf_{R\in\mathcal S^{\circ\circ}}\KL(Q\|R)>0$. However, for us to better understand this, we will defer this proof until we prove that the bipolar is closed under the total variation limits in Lemma~\ref{lem:tv-bipolar}.   

\begin{proposition}\label{prop:powerone-implies-out}
Take some $\alpha\in(0,1)$ and let $\tau$ be an $(\mathcal F_n)$-stopping time. Assume that $\sup_{P\in\Pcal}P^\infty(\tau<\infty)\le \alpha$ and that $Q^\infty(\tau<\infty)=1$. Then, it must follow that there exists a time $N<\infty$ such that for all $n\ge N$, $Q^n\notin\bPn$. 
\end{proposition}

\begin{proof}
For each $n\in\mathbb N$, let us define the event $B_n:=\{\tau\le n\}\in\mathcal F_n$. Note that $(B_n)_{n\ge 1}$ is increasing and that $\bigcup_{n\ge 1}B_n=\{\tau<\infty\}$. As a result, we have
\begin{equation}\label{eq:Bn-to-1}
Q^\infty(B_n)=Q^\infty(\tau\le n)\uparrow Q^\infty(\tau<\infty)=1.    
\end{equation}
That means there exists some $N$ such that for all $n\ge N$,
\begin{equation}\label{eq:QBn>alpha}
Q^\infty(B_n)>\alpha.   
\end{equation}
Having done this, take some $n\ge 1$ and note that for every $P\in\Pcal$, we have that $P^\infty(B_n)=P^\infty(\tau\le n)\le P^\infty(\tau<\infty)\le \alpha$. By definition $B_n\in\mathcal F_n$ and hence depends only on $(X_1,\dots,X_n)$. Therefore, we have that $P^\infty(B_n)=P^n(B_n)$. Therefore,
\begin{equation}\label{eq:supP-level}
\sup_{P\in\Pcal}P^n(B_n)\le \alpha.    
\end{equation}

Assume now for the sake of contradiction that for some $n\ge N$ we have that $Q^n\in\bPn$. By definition of $\bPn$ and $B_n\in\B^{\otimes n}$, we can use Lemma~\ref{lem:event-sup} with $\mathcal S=\Pcal^n$. Therefore, $\sup_{R\in\bPn} R(B_n)=\sup_{P\in\Pcal} P^n(B_n)$. In particular because $Q^n\in\bPn$ by assumption, we have that
\[
Q^n(B_n)\le \sup_{R\in\bPn} R(B_n)=\sup_{P\in\Pcal} P^n(B_n)\le \alpha.
\]
Note that in the last inequality, we just used \eqref{eq:supP-level}. However, we just argued that $Q^n(B_n)=Q^\infty(B_n)>\alpha$ by \eqref{eq:QBn>alpha}. This is a clear contradiction. Hence it certainly must be the case that $Q^n\notin\bPn$ for all $n\ge N$. This completes the proof.
\end{proof}

\subsection{Necessary and Sufficient Conditions for Sequential Testability}

Our next corollary will show how non-membership in the bipolar actually propagates forward in $n$. However, to get there and make the proof more straightforward, we will first present a lemma to show that bipolar membership is preserved under marginalization.

\begin{lemma}\label{lem:bipolar-marginal}
Let $1\le n\le k$ and let $\pi_{1:n}:\X^k\to \X^n$ be the projection onto the first $n$ coordinates. If $R\in(\Pcal^k)^{\circ\circ}$, then the marginal
\(
R_{1:n}:=R\circ \pi_{1:n}^{-1}\in(\Pcal^n)^{\circ\circ}.
\)
\end{lemma}

\begin{proof}
Consider some $E\in(\Pcal^n)^\circ$. Define the map $\widetilde E:\X^k\to[0,\infty]$ by $\widetilde E(x_1,\dots,x_k):=E(x_1,\dots,x_n)$. If we use the product form $P^k=P^n\otimes P^{k-n}$ we have that for any $P\in\Pcal$, $\E_{P^k}[\widetilde E]=\E_{P^n}[E]\le 1$. By Definition~\ref{def:bipolar} we have that $\widetilde E\in(\Pcal^k)^\circ$. Since by assumption $R\in(\Pcal^k)^{\circ\circ}$, $\E_R[\widetilde E]\le 1$. But notice how $\E_R[\widetilde E]=\E_{R_{1:n}}[E]$ by definition of the marginal. Thus, $\E_{R_{1:n}}[E]\le 1$ for every $E\in(\Pcal^n)^\circ$, which is exactly in fact the condition that $R_{1:n}\in(\Pcal^n)^{\circ\circ}$, completing the proof.     
\end{proof}

Now, with this lemma shown, the proof of the corollary that nonmembership propagates forward in $n$ will be fairly simple.

\begin{corollary}\label{cor:out-propogates}
If $Q^n\not\in\bPn$ for some $n$, then for every $k\ge n$, we have that $Q^k\not\in\bPk$.    
\end{corollary}

\begin{proof}
Take some $k\ge n$. Let us assume to the contrary that $Q^k\in\bPk$. Thanks to Lemma~\ref{lem:bipolar-marginal}, we have that the marginal of $Q^k$ on the first $n$ coordinates belongs to $\bPn$. However, that marginal is exactly $Q^n$. So that would mean that $Q^n\in\bPn$, which contradicts the assumption.    
\end{proof}

\noindent Finally, the last result we will present will specifically show that if we have weak compactness in conjunction with the $\KLinf>0$, that necessarily implies eventual non-membership also.

\begin{proposition}\label{prop:compact-implies-out-eventual}
Assume that $\Pcal$ is weakly compact and that $d:=\KLinf(Q,\Pcal)>0$. Then, there exists some $N<\infty$ such that for all $n\ge N$, $Q^n\notin\bPn$.   
\end{proposition}

\begin{proof}
Because of the fact that $\Pcal$ is weakly compact by assumption, we can immediately use the fact that \parencite[Lemma~3]{ram2026powersequentialtestsexist} tells us that $\Phi(R):=\inf_{P\in\Pcal}\KL(R\|P)$ is weakly lower semicontinuous everywhere. And in particular, it is at $Q$. Now, since $d=\Phi(Q)\in(0,
\infty)$, thanks to Theorem~\ref{thm:main} we have $\lim_{n\to\infty}\frac{1}{n}a_n=d$. Let us now choose $N$ such that for all $n\ge N$, we have that $a_n/n \ge d/2$. Therefore,
\begin{equation}\label{eq:an-positive}
a_n\ge \frac{d}{2}n>0.    
\end{equation}

\noindent Now consider some $n\ge N$. If $Q^n\in\bPn$, that would mean that $Q^n$ is feasible in the infimum defining $a_n$. Therefore that means that $a_n=\inf_{R\in\bPn}\KL(Q^n\|R)\le \KL(Q^n\|Q^n)=0$. But this contradicts \eqref{eq:an-positive}. So, $Q^n\notin\bPn$ for all $n\ge N$.
\end{proof}

All that we have developed so far leads to this theorem on necessary and sufficient characterizations of existence of power-one sequential tests for arbitrary composite nulls versus point alternatives.

\begin{theorem}\label{thm:powerone-characterization}
The following statements are equivalent.
\begin{enumerate}
    \item For every $\alpha\in(0,1)$, there exists a level-$\alpha$ sequential test of $\Pcal$ against $Q$ with power one.
    \item There exists some $k\in\mathbb N$ such that $Q^k \notin (\Pcal^k)^{\circ\circ}$.
    \item There exists some $k\in\mathbb N$ such that $a_k(Q)>0$.
    \item $\lim_{n\to\infty}\frac{1}{n}\inf_{R\in(\Pcal^n)^{\circ\circ}}\KL(Q^n\|R)>0$.
\end{enumerate}
And moreover, if these statements hold and $k$ is as in statement \textcolor{red}{(2)} above, for every $n\ge k$, $Q^n\notin (\Pcal^n)^{\circ\circ}$ and $a_n(Q)>0$.
\end{theorem}

An important point to make is that in Theorem~\ref{thm:powerone-characterization}, we do not require the one step condition, i.e.\ $a_1(Q)>0$. We can see this with a counter-example as follows: this is sufficient, but not necessary. 

\begin{example}\label{ex:onestep-suf}
Suppose that $\X=\{0,1\}$. Let $\Pcal=\{\delta_0, \delta_1\}$ and $Q=\frac12 \delta_0 + \frac12 \delta_1$. Then $a_1(Q):=\inf_{R\in\Pcal^{\circ\circ}}\KL(Q\|R)=0$ since $Q\in\Pcal^{\circ\circ}$. However, suppose that $A:=\{(0,1),(1,0)\}\subseteq \X^2$. Then, it follows that $\sup_{P\in\Pcal}P^{2}(A)=0$, while $Q^{2}(A)=1/2$. Thanks to Lemma~\ref{lem:event-sup}, it follows that $Q^2\notin (\Pcal^2)^{\circ\circ}$, and hence $a_2(Q)>0$. Therefore, thanks to Theorem~\ref{thm:powerone-characterization}, a power-one sequential test exists. Therefore, the correct assumption should be that there is positivity at some horizon $k$, namely $d_\star(Q,\Pcal)>0$.    
\end{example}

\section{When the Bipolar Rate Equals $\KLinf$}\label{sec:bipeqklinf}

For each $n$, let us define the class $\Pcal^n:=\{P^n:P\in\Pcal\}$ on $(\X^n, \B^{\otimes n})$. Now let $(\Pcal^n)^{\circ\circ}$ be the bipolar. Now define
\[
\infKL(Q^n,\Pcal^n) := \inf_{R\in(\Pcal^n)^{\circ\circ}} \KL(Q^n\|R)\in[0,\infty].
\]
Thanks to \parencite{larsson2025numeraire}, the infimum is always attained. Thus we can denote any minimizer as $P_n^\star$ and call it therefore the RIPr of $Q^n$ onto $(\Pcal^n)^{\circ\circ}$. We will now present our main result of when the $\infKL$ and $\KLinf$ coincide.

\begin{theorem}\label{thm:main}
Suppose that $Q\in\mathcal{M}_1(\X)$. Let $\Pcal\subseteq\mathcal{M}_1(\X)$ be nonempty. Assume that $\KLinf(Q, \Pcal):=\inf_{P\in\Pcal}\KL(Q\|P)<\infty$ and that $\Pcal$ is $\KLinf$ lower semicontinuous at $Q$. Then it follows that
\[
\boxed{
\lim_{n\to\infty}\frac{1}{n}\infKL(Q^n,\Pcal^n)=\KL_{\inf}(Q,\Pcal).
}
\]
And, this conclusion will hold whenever $\Pcal$ is weakly compact.
\end{theorem}

To prove this theorem, we will use a KL lower bound that follows from the data processing inequality on a binary $\sigma$-field. 
\begin{lemma}\label{lem:binary-dpi}
Let $(\Omega, \mathcal{F})$ be measurable and let $\mu, \nu\in\mathcal{M}_1(\Omega)$ and let $A\in\mathcal{F}$. Now, define $p:=\mu(A)$ and $q:=\nu(A)$. Then it follows that
\[
\KL(\mu\|\nu)\ge p\log\frac{p}{q}+(1-p)\log\frac{1-p}{1-q}.
\]
\end{lemma}

In addition, an important lemma for us will be a convex-set empirical upper bound. Note that this is essentially a nonasymptotic version of Sanov, and we can attribute it to Csisz{\'a}r \cite{csiszar1984sanov}. Importantly, his assumption of almost complete convexity holds with our assumptions. It is crucial to note that under convexity alone, the lemma would not hold, as explained by Csisz{\'a}r in his work.

\begin{lemma}\label{lem:convex-empirical}
Let $P\in\mathcal M_1(\X)$ and let $C\subseteq\mathcal M_1(\X)$ be convex and weakly closed. Define $A_n:=\{\widehat Q_n\in C\}\subseteq \X^n$. Then it follows that for every $n\ge 1$,
\[
P^n(A_n)\ \le\ \exp\Bigl(-n\inf_{R\in C}\KL(R\|P)\Bigr).
\]
\end{lemma}

\begin{proof}[Proof of Theorem~\ref{thm:main}]
Write the shorthand $d:=\KLinf(Q,\Pcal)=\inf_{P\in\Pcal}\KL(Q\|P)\in[0,\infty)$. Note that if $d=0$, then for every $n$ our universal lower upper bound below will imply that $0\le \infKL(Q^n, \Pcal^n)\le nd=0$. So it would follow that $\infKL(Q^n, \Pcal^n)=0$ for all $n$ and so our desired limit would hold. Hence we can assume going forward without loss of generality that $d>0$. We will first prove universal upper bound. Note that this holds under no assumptions and holds for any measurable space (we do not need to assume Polish structure). Now, because $\Pcal^n\subseteq (\Pcal^n)^{\circ\circ}$,
\[
\infKL(Q^n, \Pcal^n) = \inf_{R\in(\Pcal^n)^{\circ\circ}}\KL(Q^n\|R)\le \inf_{P\in\Pcal} \KL(Q^n\|P^n).
\]
We also have that
\[\inf_{P\in\Pcal}\KL(Q^n\|P^n)=\inf_{P\in\Pcal}n\KL(Q\|P)= n\inf_{P\in\Pcal}\KL(Q\|P),
\]
Therefore, for all $n$ we have that $\infKL(Q^n, \Pcal^n)\le nd$, and hence
\[
\limsup_{n\to\infty}\frac{1}{n}\infKL(Q^n, \Pcal^n)\le d.
\]
We now show the matching lower bound to show tightness. Take a particular $\varepsilon\in\bigl(0,\tfrac{d}{4}\bigr)$. Now define $\Phi(R):=\inf_{P\in\Pcal}\KL(R\|P)$. Note that our second assumption in this theorem entails that $\Phi$ be lower semicontinuous at $Q$. Or equivalently, for this specific $\varepsilon$, there exists a weakly open neighborhood $O_\varepsilon$ of $Q$ such that $\Phi(R)\ge d-\varepsilon$ for all $R\in O_\varepsilon$. Now consider a particular bounded Lipschitz metric $d_{\mathrm{BL}}$ on $\mathcal M_1(\X)$, which necessarily metrizes the weak topology on Polish $\X$. Because $O_\varepsilon$ is weakly open and contains $Q$, there exists $r>0$ such that the open $d_{\mathrm{BL}}$-ball
\[
B_r^{\mathrm{BL}}(Q):=\{R\in\mathcal M_1(\X):d_{\mathrm{BL}}(R,Q)<r\}
\]
is contained in $O_\varepsilon$. In addition set $\delta:=r/2$ and define the closed ball
\[
C:=\overline{B}_\delta^{\mathrm{BL}}(Q):=\{R\in\mathcal M_1(\X): d_{\mathrm{BL}}(R,Q)\le \delta\}.
\]
Clearly, then $C$ is weakly closed. And because $d_{\mathrm{BL}}$ is induced by a norm on signed measures, it follows that $C$ is convex. In addition, $C\subseteq B_r^{\mathrm{BL}}(Q)\subseteq O_\varepsilon$, so $\inf_{R\in C}\Phi(R)\ge d-\varepsilon$. Define the event $A_n:=\{\widehat Q_n\in C\}\subseteq \X^n$. Since $d_{\mathrm{BL}}(\widehat Q_n,Q)\to 0$ almost surely under $Q^\infty$, it follows that $Q^n(A_n)\to 1$. Define $\beta_n:=\sup_{P\in\Pcal}P^n(A_n)$. Thanks to Lemma~\ref{lem:convex-empirical}, for each $P\in\Pcal$, $P^n(A_n)\le \exp\Bigl(-n\inf_{R\in C}\KL(R\|P)\Bigr)$. Therefore taking the supremum, we have 
\[
\beta_n\le \exp\Bigl(-n\inf_{R\in C}\inf_{P\in\Pcal}\KL(R\|P)\Bigr) =\exp\Bigl(-n\inf_{R\in C}\Phi(R)\Bigr) \le \exp\bigl(-n(d-\varepsilon)\bigr).
\]

\noindent Equivalently for all $n\ge 1$, 
\begin{equation}\label{eq:beta-bound}
\frac1n\log\frac1{\beta_n}\ge d-\varepsilon.    
\end{equation}

\noindent Now take any particular $n\ge1$ and any $R\in(\Pcal^n)^{\circ\circ}$. By Lemma~\ref{lem:event-sup} applied to $\mathcal{S}=\Pcal^n$ and the event $A_n$ we have that $R(A_n)\le \sup_{S\in\Pcal^n} S(A_n)=\sup_{P\in\Pcal}P^n(A_n)=\beta_n$. Now let $p_n:=Q^n(A_n)$ and $q_n:=R(A_n)$. Then $p_n\to 1$ and $q_n\le \beta_n$. Thus by Lemma~\ref{lem:binary-dpi} we get that
\begin{equation}\label{eq:kl-binary}
\KL(Q^n\|R)\ge p_n\log\frac{p_n}{q_n}+(1-p_n)\log\frac{1-p_n}{1-q_n}.   
\end{equation}
We can use the fact that $\log(1-q_n)\le 0$ and therefore that $-(1-p_n)\log(1-q_n)\ge 0$ to give us a simpler bound. That is,
\begin{align}
p_n\log\frac{p_n}{q_n}+(1-p_n)\log\frac{1-p_n}{1-q_n} &=p_n\log p_n + (1-p_n)\log(1-p_n)-p_n\log q_n-(1-p_n)\log(1-q_n)\nonumber\\
&\ge -H(p_n)+p_n\log\frac{1}{q_n},\label{eq:kl-lower-entropy}
\end{align}
where $H(p):=-p\log p-(1-p)\log(1-p)$ is the binary entropy. Now since $p_n\to 1$, it must be the case that $H(p_n)\to 0$. Also, since $q_n\le \beta_n$, it must be the case that $\log(1/q_n)\ge \log(1/\beta_n)$. Together then both \eqref{eq:kl-binary} and \eqref{eq:kl-lower-entropy} give us that
\[
\KL(Q^n\|R)\ge -H(p_n)+p_n\log\frac{1}{\beta_n}.
\]
Dividing by $n$ and taking the $\liminf_{n\to\infty}$, and taking the infimum over $R\in(\Pcal^n)^{\circ\circ}$ gives us for all $n\ge n_0(\varepsilon)$ therefore that $\infKL(Q^n, \Pcal^n)\ge -H(p_n)+p_n\log\frac{1}{\beta_n}$. Now let $c_n:=\frac1n\log\frac1{\beta_n}\ge 0$, which is because $\beta_n\le 1$. Then it follows that
\[
\frac1n\infKL(Q^n,\Pcal^n)\ge -\frac{H(p_n)}{n}+p_n c_n.
\]
Now $H(p_n)\le \log 2$. Therefore, $H(p_n)/n\to0$, and $p_n\to 1$. Now by \eqref{eq:beta-bound}, we have that $c_n=\frac1n\log\frac1{\beta_n}\ge d-\varepsilon$ for all $n$. Therefore, $\liminf_{n\to\infty} c_n\ge d-\varepsilon$. Since $c_n\ge 0$ and $p_n\to 1$, it follows that $\liminf p_nc_n \ge \liminf c_n$. Therefore, we get
\[
\liminf_{n\to\infty}\frac1n\infKL(Q^n,\Pcal^n)\ge d-\varepsilon.
\]
Recall that $\varepsilon>0$ was arbitrary. So, $\liminf_{n\to\infty}\frac{1}{n}\infKL(Q^n, \Pcal^n)\ge d$, completing the proof.
\end{proof}

One thing to note is why we need some regularity assumption like lower semicontinuity of $R\mapsto \inf_{P\in\Pcal}\KL(R\|P)$ at $Q$. If we did not have this, then there could be a sequence $R_k\Rightarrow Q$ with $\inf_{P\in\Pcal}\KL(R_k\|P)$ much smaller than $\inf_{P\in\Pcal}\KL(Q\|P)$. Suppose that we are in one of these cases. Then we can easily build very rich classes $\Pcal$ which are noncompact and very oscillatory such that the bipolar $(\Pcal^n)^{\circ\circ}$ becomes large enough to where $\infKL(Q^n, \Pcal^n)$ grows sublinearly in n, while the $\KLinf(Q, \Pcal)$ stays strictly postive. Hence our desired asymptotic equality would fail. This case of failure is the motivation for weak compactness of $\Pcal$ or any other condition that implies our lower semicontinuity assumption.

Let us now once again compare the two quantities $\infKL(Q,\mathcal P):=\inf_{R\in\Pcal^{\circ\circ}}\KL(Q\|R)$ and $\KLinf(Q,\mathcal P):=\inf_{P\in\mathcal P}\KL(Q\|P)$. We will write,
\[
\mathcal B_b(\X,\B):=\{f:\X\to\mathbb R: f \text{ is bounded and }\B\text{-measurable}\}.
\]
In addition, we will use the notation $L(Q, f):=\E_Q[f]-\Lambda_{\mathcal P}(f)$ and $\Lambda_{\Pcal}(f):=\sup_{P\in\Pcal}\log \E_P[e^f]$. We will use a shorthand for the very well known $\mathrm{GRO}$ value as $R_{\mathcal P}(Q)$. Of course, \parencite{larsson2025numeraire} show that the $\infKL(Q,\mathcal P)=R_{\mathcal P}(Q)$. We now show that $R_{\mathcal P}(Q)=\sup_{f\in\mathcal B_b(\X,\B)}L(Q,f)$. 

Now for some particular $Q\in\mathcal M_1$, let us define $$R_{\mathcal{P}}(Q):=\sup_{X\in\mathcal{E}(\mathcal{P})} \E_{Q}[\log X]\in[0,+\infty].$$ 
Note that though $R_{\mathcal P}(Q)$ is defined by optimizing over all the $e$-variables, we can represent it by bounded measurable functions. We can have a bounded variational representation of the point alternative value. To this end, for $f\in\Bb$ define,
\[
\Lambda_{\mathcal{P}}(f):=\log\Bigl(\sup_{P\in\mathcal{P}}\E_P[e^f]\Bigr),\quad
L(Q,f):=\E_{Q}[f]-\Lambda_{\mathcal{P}}(f).
\]
Since $f$ is bounded, $e^f$ is bounded. Therefore, $\Lambda_{\mathcal P}(f)\in\mathbb R$ is well-defined.

\begin{theorem}\label{thm:bounded-representation}
For every $Q\in\mathcal M_1$, we have $R_{\mathcal P}(Q)=\sup_{f\in\Bb} L(Q,f)$. More explicitly,
\[
\sup_{X\in\mathcal{E}(\mathcal{P})} \E_{Q}[\log X] = \sup_{f\in\Bb} \E_{Q}[f] - \log\Bigl(\sup_{P\in\mathcal{P}}\E_P[e^f]\Bigr) = \inf_{P\in\mathcal{P}^{\mathrm{Eff}}} H(Q\mid P).
\]
\end{theorem}

\begin{proof}
First, we start with the simpler direction; we will show that $\sup_{f\in\Bb}L(Q,f)\le R_{\mathcal P}(Q)$. To begin, take an arbitrarily chosen $f\in\Bb$ and set $X_f:=e^{f-\Lambda_{\mathcal P}(f)}$. It follows that for every $P\in\mathcal P$ we have 
\[
\E_P[X_f]=\E_P\left[e^{f-\Lambda_{\mathcal{P}}(f)}\right] = e^{-\Lambda_{\mathcal{P}}(f)}\, \E_P[e^f]\le e^{-\Lambda_{\mathcal{P}}(f)}\, \sup_{P'\in\mathcal{P}} \E_{P'}[e^f]=1.
\]
Thus, $X_f\in\mathcal E(\mathcal P)$, implying 
\begin{align*}
L(Q,f)&=\E_{Q}[f]-\Lambda_{\mathcal{P}}(f)\\
&=\E_{Q}\left[\log\bigl(e^{f-\Lambda_{\mathcal{P}}(f)}\bigr)\right]\\
&=\E_{Q}[\log X_f]\\
&\le \sup_{X\in\mathcal{E}(\mathcal{P})} \E_{Q}[\log X]\\
&=R_{\mathcal{P}}(Q).
\end{align*}
Taking the supremum over all $f\in\Bb$ gives us that $\sup_{f\in\Bb}L(Q,f)\le R_{\mathcal P}(Q)$. So the first direction is proved. 

It remains to now show that $R_{\mathcal{P}}(Q)\le \sup_{f\in\Bb}L(Q,f)$. To this end, let us take some particular but arbitrary $X\in\mathcal E(\mathcal P)$. We need to show that $\E_{Q}[\log X]\le \sup_{f\in\Bb}L(Q,f)$. First suppose that $E_Q[(\log X)^-]=\infty$. In this case, by convention it follows that $\E_{Q}[\log X]=-\infty$, so the bound follows. Thus, let us assume wlog from this point on that $\E_{Q}[(\log X)^-]<\infty$. For each $n\in\mathbb N$, let us define the bounded truncation $f_n:=(\log X\wedge n)\vee (-n)\in\Bb$. Therefore, it follows that,
\[
e^{f_n}=\exp\bigl((\log X\wedge n)\vee(-n)\bigr)=(X\wedge e^n)\vee e^{-n}\le X+e^{-n},
\]
which holds pointwise on $\Omega$. Therefore, since $X\in\mathcal E(\mathcal P)$, we have that for every $P\in\mathcal P$, 
\[
\E_P[e^{f_n}]\le \E_P[X+e^{-n}] = \E_P[X]+e^{-n}\le 1 +e^{-n}.
\]
Taking the supremum over all $P\in\mathcal P$ and then taking the logarithm gives us
\begin{equation}\label{eq:Lambda-bound}
\Lambda_{\mathcal P}(f_n)\le \log(1+e^{-n}).    
\end{equation}
On the other hand, the positive and negative parts of $f_n$ are simply $f_n^+=(\log X)^+\wedge n$ and $f_n^-=(\log X)^-\wedge n$ respectively. By the monotone convergence theorem, it follows that $\E_{Q}[f_n^+]\uparrow\E_{Q}[(\log X)^+]$ and $\E_{Q}[f_n^-]\uparrow\E_{Q}[(\log X)^-]<\infty$ (we assumed earlier wlog that the negative limit is finite). And so therefore it is appropriate for us to subtract these limits. Therefore we get that $\E_{Q}[f_n]=\E_{Q}[f_n^+]-\E_{Q}[f_n^-]\to \E_{Q}[(\log X)^+]-\E_{Q}[(\log X)^-]=\E_{Q}[\log X]$. Hence, we conclude that
\begin{equation}\label{eq:fn-to-logX}
\E_{Q}[f_n] \longrightarrow \E_{Q}[\log X].    
\end{equation}

Now \eqref{eq:Lambda-bound} and \eqref{eq:fn-to-logX} yield
\[
L(Q,f_n)=\E_{Q}[f_n]-\Lambda_{\mathcal{P}}(f_n)\ge \E_{Q}[f_n]-\log(1+e^{-n}) \to \E_{Q}[\log X]. 
\]
Since $f_n \in \Bb$, we therefore get
\begin{align*}
\sup_{f\in\Bb} L(Q,f) &\ge  \E_{Q}[\log X].
\end{align*}

This holds for every $e$-variable $X$. So, it follows that $R_{\mathcal{P}}(Q)=\sup_{X\in\mathcal{E}(\mathcal{P})}\E_{Q}[\log X]\le\sup_{f\in\Bb}L(Q,f)$, concluding the proof.
\end{proof}

It is interesting to note how the $\infKL$ and $\KLinf$ relate nonasymptotically. We will show in the following proposition.

\begin{proposition}\label{prop:infkl-klinf}
For every class $\Pcal$ of probability measures and every probability measure $Q$, we always have that
\[
\infKL(Q,\Pcal)\le \KLinf(Q,\Pcal).
\]
If in addition, $(\X,\B)$ is the Borel space of a metric space and $\Pcal$ is weakly compact and convex, then
\[
\infKL(Q,\Pcal)=\KLinf(Q,\Pcal).
\]
\end{proposition}

\begin{proof}
We will first show the first part that $\infKL(Q, \mathcal P)\le \KLinf(Q,\mathcal P)$ always holds. We first show that $\Pcal\subseteq\Pcal^{\circ\circ}$. Suppose that $P\in\mathcal P$. To show that $P\in\Pcal^{\circ\circ}$, we need to only show that $\E_P[X]\le 1$ for every $X\in\Pcal^{\circ}$. But, this is just the definition of $X\in\Pcal^{\circ}$. Therefore, $P\in\Pcal^{\circ\circ}$, and hence $\Pcal\subseteq\Pcal^{\circ\circ}$. Now let us take the infimum of the $\KL(Q\|\cdot)$ over both sets, yielding 
\[
\infKL(Q,\Pcal) = \inf_{R\in\Pcal^{\circ\circ}}\KL(Q\|R)\le \inf_{P\in\Pcal}\KL(Q\|P) = \KLinf(Q,\mathcal P).
\]
It remains to now show that if $\Pcal$ is weakly compact and convex, then the reverse inequality will hold also. To this end, assume $(\X,\B)$ is the Borel space of a metric space and that $\Pcal$ is weakly compact and convex. Thanks to \parencite{shekhar2025optimalanytimevalidtestscomposite} we know that
\begin{equation}\label{shubanshu-result}
\KLinf(Q,\Pcal)=\sup_{f\in C_b(\X)}\left\{\E_Q[f]-\sup_{P\in\Pcal}\log \E_P[e^f]\right\} = \sup_{f\in C_b(\X)}L(Q,f).
\end{equation}
On the other hand, using the fact that $\infKL(Q,\Pcal)=R_{\Pcal}(Q)$ by our previous bounded-representation result we have that
\begin{equation}\label{eq:bounded-measurable-dual}
\infKL(Q,\Pcal)=R_{\Pcal}(Q)=\sup_{f\in\mathcal B_b(\X,\B)}L(Q,f).
\end{equation}

Every bounded continuous function is bounded and measurable so $C_b(\X)\subseteq\mathcal B_b(\X,\B)$. This enlargement of the function classes can only increase the supremum. Therefore, from \eqref{shubanshu-result} and \eqref{eq:bounded-measurable-dual},
\[
\KLinf(Q,\Pcal)=\sup_{f\in C_b(\X)}L(Q,f)\le \sup_{f\in \mathcal B_b(\X,\B)}L(Q,f)=\infKL(Q,\Pcal).
\]
And, taken together with the bound where we lower bounded the $\KLinf$, we get that $\infKL(Q,\Pcal)=\KLinf(Q,\Pcal)$. This completes the proof.
\end{proof}

\section{Counterexamples and Structural Examples}\label{sec:counterx}
We will first see what happens without regularity. We will show a counterexample of what happens when the second assumption in our Theorem~\ref{thm:main} fails. But we will first prove that total variation limits lie in the bipolar. This will help us better intuit the counter-example.

\begin{lemma}\label{lem:tv-bipolar}
Let $(\Omega, \mathcal F)$ be measurable and let $\mathcal S\subseteq \mathcal M_1(\Omega)$ be nonempty. Suppose that in fact there exists a sequence $(S_m)_{m\ge 1}\subseteq\mathcal S$ such that $S_m\to S$ in total variation. Meaning, $\|S_m-S\|_{\mathrm{TV}}:=\sup_{A\in\mathcal F}|S_m(A)-S(A)|\longrightarrow 0$. Then, we have that $S\in\mathcal S^{\circ\circ}$.  
\end{lemma}

\begin{proof}
Let $E\in\mathcal S^\circ$. Then, $\E_{S_m}[E]\le 1$ for all $m$. Now, for $M\ge 1$ define the truncation $E^{(M)}:=E\wedge M$ so that both $0\le E^{(M)}\le M$ and $\E_{S_m}[E^{(M)}]\le \E_{S_m}[E]\le 1$. Total variation convergence means convergence of the integrals of bounded measurable functions. Hence we get that
\[
\E_S[E^{(M)}]=\lim_{m\to\infty}\E_{S_m}[E^{(M)}]\le 1.
\]
Now letting $M\to\infty$ and applying the monotone convergence theorem, we can see that $\E_S[E]\le 1$. Because this holds for all $E\in\mathcal S^{\circ}$, it indeed must be the case that $S\in\mathcal S^{\circ\circ}$.
\end{proof}

We now present another proposition that entails that the bipolar contains the total-variation closed convex hull. It's immediate from our Lemma~\ref{lem:tv-bipolar} and from Theorem~3.4 of \cite{larsson2026completecharacterizationtestablehypotheses}.

\begin{proposition}\label{prop:bipolar-tvconv}
Let $(\Omega, \mathcal F)$ be measurable and let $\mathcal S\subseteq \mathcal M_1(\Omega)$ be nonempty. Denote $\operatorname{conv}(\mathcal S)$ for its convex hull, the set of all finite convex combinations of elements of $\mathcal S$. 
Let $\|\cdot\|_{\mathrm{TV}}$ denote the total variation distance. Then, it follows that $\mathcal S^{\circ\circ}$ is convex and closed under $\|\cdot\|_{\mathrm{TV}}$-limits. Let us denote $\overline{(\cdot)}^{\mathrm{TV}}$ to mean closure in total variation. Then we have that, $\overline{\operatorname{conv}(\mathcal{S})}^{\mathrm{TV}}\subseteq \mathcal S^{\circ\circ}$.
\end{proposition}

Now, we will present the proof of we deferred earlier: non-membership in the bipolar implies a strictly positive $\KL$ gap. 

\begin{proposition}\label{prop:notinbipolar-positive}
Suppose that $(\Omega, \mathcal F)$ is measurable. Let $\mathcal S\subseteq \mathcal M_1(\Omega)$ be nonempty. And let $Q\in\mathcal M_1(\Omega)$. If $Q\notin\mathcal S^{\circ\circ}$, then it follows that $\inf_{R\in\mathcal S^{\circ\circ}}\KL(Q\|R)>0$.   
\end{proposition}

\begin{proof}
Thanks to Proposition~\ref{prop:bipolar-tvconv}, it follows that $\mathcal S^{\circ\circ}$ is closed in total-variation limits. This means that it is a closed subset of the metric space $(\mathcal M_1(\Omega),\|\cdot\|_{\mathrm{TV}})$. By definition, $Q\not\in\mathcal S^{\circ\circ}$, so the $\mathrm{TV}$-distance from $Q$ to $\mathcal S^{\circ\circ}$ is strictly positive. Namely, $\delta:=\inf_{R\in\mathcal S^{\circ\circ}}\|Q-R\|_{\mathrm{TV}}>0$. Therefore by Pinsker's inequality for every $R\in\mathcal S^{\circ\circ}$, $\KL(Q\|R)\ge 2\|Q-R\|_{\mathrm{TV}}^2\ge 2\delta^2$. Then, taking the infimum over $R\in\mathcal S^{\circ\circ}$ proves the claim.
\end{proof}

With all this proved, we are now ready to present our proposition. Specifically we will show failure of Theorem~\ref{thm:main} when the second assumption fails.

\begin{proposition}\label{prop:counterexample}
Let $\X=[0,1]$ with its Borel $\sigma$-algebra and let $Q$ be the uniform distribution on $[0,1]$. For $k\in\mathbb N$, let us define $\varepsilon_k:=2^{-k}$, $A_k:=[0,\varepsilon_k]$, and $a_k:=\exp(-1/\varepsilon_k)=\exp(-2^k)$. Further, define $b_k:=\frac{1-\varepsilon_k a_k}{1-\varepsilon_k}\in(1,\infty)$ and $\frac{dP_k}{dQ}:=a_k\one_{A_k}+b_k\one_{A_k^c}$. Now let $\Pcal:=\{P_k:k\in\mathbb N\}$. Then the following hold.
\begin{enumerate}
    \item First, $\KLinf(Q,\Pcal)>0$. Indeed, notice that $\KLinf(Q,\Pcal)\ge 1+\tfrac12\log\tfrac12>1/2.$
    \item   Second, $\Phi(R):=\inf_{P\in\mathcal P}\KL(R\|P)$ is not lower semicontinuous at $Q$ for the weak topology. Thus, Theorem~\ref{thm:main}'s second assumption fails.
    \item Third, for every $n\ge 1$, we have that $Q^n\in(\Pcal^n)^{\circ\circ}$ and thus $\infKL(Q^n, \Pcal^n)=0$.
\end{enumerate}
As a consequence, the KL rate identity in Theorem~\ref{thm:main} unfortunately fails. That is,
\[
\boxed{
\lim_{n\to\infty}\frac1n\infKL(Q^n,\Pcal^n)=0\quad\text{while}\quad\KL_{\inf}(Q,\Pcal)>0.
}
\]
\end{proposition}

\begin{proof}
We will begin our proof by first showing that $P_k\Rightarrow Q$ and even $P_k\to Q$ in total variation. Note that $dP_k/dQ$ is $Q$-ae equal to $a_k$ on $A_k$ and $b_k$ on $A_k^c$. Therefore we get that
\begin{align*}
\|P_k-Q\|_{\mathrm{TV}}&=\frac12\int_{[0,1]}\Bigl|\frac{dP_k}{dQ}-1\Bigr|dQ\\
&=\frac12\Bigl(\varepsilon_k(1-a_k)+(1-\varepsilon_k)(b_k-1)\Bigr)\\
&=\frac12\Bigl(\varepsilon_k(1-a_k)+\varepsilon_k(1-a_k)\Bigr) =\varepsilon_k(1-a_k)\le\varepsilon_k\xrightarrow[k\to\infty]{}0.
\end{align*}
Specifically, $P_k\Rightarrow Q$. We now show that $\KLinf(Q,\Pcal)>0$. For each $k$,
\begin{align*}
\KL(Q\|P_k)&=\int\log\Bigl(\frac{dQ}{dP_k}\Bigr)dQ\\
&=\varepsilon_k\log\frac{1}{a_k}+(1-\varepsilon_k)\log\frac{1}{b_k}\\
&=1+(1-\varepsilon_k)\log\Bigl(\frac{1-\varepsilon_k}{1-\varepsilon_k a_k}\Bigr).
\end{align*}
Note that $a_k\le 1$ and because of this, we have that $1-\varepsilon_k a_k\le 1$. Therefore we get that
\[
\log\Bigl(\frac{1-\varepsilon_k}{1-\varepsilon_k a_k}\Bigr)\ge\log(1-\varepsilon_k).
\]
So clearly $\KL(Q\|P_k)\ge 1+(1-\varepsilon_k)\log(1-\varepsilon_k)$. Note that $\varepsilon_k\in(0,1/2]$. And, the function $g(\varepsilon):=1+(1-\varepsilon)\log(1-\varepsilon)$ is decreasing on $(0,1/2]$. Because $g'(\varepsilon)=-\log(1-\varepsilon)-1<0$. Hence it follows that the minimum over $k\in\mathbb N$ is attained at $\varepsilon_1=1/2$. This gives us that $\KLinf(Q,\Pcal)=\inf_{k\in\mathbb N}\KL(Q\|P_k)\ge 1+\frac12\log\frac12 >\frac12$. Hence the $\KLinf$ is greater than $0$. We now prove that assumption two fails. Define the same $\Phi(R):=\inf_{P\in\Pcal}\KL(R\|P)$. For every $k$, choosing $P=P_k$ gives us $\Phi(P_k)=0$. However, $P_k\Rightarrow Q$ while $\Phi(Q)=\KLinf(Q,\Pcal)>0$ as we just showed. So $\Phi$ cannot be lower semicontinuous at $Q$. Let us now finish our proof by showing that $Q^n\in(\Pcal^n)^{\circ\circ}$ and $\infKL(Q^n, \Pcal^n)=0$. Take some arbitrary but particular $n\ge 1$. We have
\[
\|P_k^n-Q^n\|_{\mathrm{TV}}\le 1-(1-\|P_k-Q\|_{\mathrm{TV}})^n \le n\|P_k-Q\|_{\mathrm{TV}} \xrightarrow[k\to\infty]{}0.
\]
So $P_k^n\to Q^n$ in total variation. So we can use Lemma~\ref{lem:tv-bipolar} respectively with $\mathcal S=\Pcal^n$, giving $Q^n\in(\Pcal^n)^{\circ\circ}$. Therefore we get that
\[
\infKL(Q^n,\Pcal^n) =\inf_{R\in(\Pcal^n)^{\circ\circ}}\KL(Q^n\|R) \le \KL(Q^n\|Q^n)=0.
\]
So, we have that $\inf\text{-}\KL(Q^n,\Pcal^n)=0$ for all $n$, hence the rate identity in the theorem will fail.
\end{proof}

We will now revisit the example from \parencite{ram2026powersequentialtestsexist}. We will now make clear why weak lower semicontinuity is sufficient, but not necessary to get our rate identity also.

\begin{proposition}\label{prop:weaklsc-not-necessary}
Let $\X=[0,1]$ on the Borel $\sigma$-algebra. Let $Q=\delta_0$ and for $k\in\mathbb N$, define $P_k:=\frac12\delta_0 +\frac12\delta_{1/k}$. Then let $\Pcal:=\{P_k:k\in\mathbb N\}$. Then,
\begin{enumerate}
    \item $\KLinf(Q,\Pcal)=\log 2>0$.
    \item $\Phi(R)=\inf_{P\in\Pcal}\KL(R\|P)$ is not weakly lower semicontinuous at $Q$.
    \item For every $n$, $\infKL(Q^n, \Pcal^n)=n\log 2$, and  $\lim_{n\to\infty}\frac1n a_n=\log 2 = \KLinf(Q,\Pcal)$.
    \item For every $\alpha\in(0,1)$, there exists a level $\alpha$ sequential test with power-one against $Q$.
    \item There exists an e-process $(E_n)_{n\ge 0}$ on $\bbF$ for $\Pcal$ with exact growth rate $\log 2$ under $Q$.
\end{enumerate}
\end{proposition}

\begin{proof}
All other statements save for the third was proven in \parencite{ram2026powersequentialtestsexist}. We will prove that the rate identity still holds, however. Consider some particular $n$ and define $A_n:=\{x^n:x_1=\cdots=x_n=0\}$. It must be the case that $Q^n(A_n)=1$ and for every $k$, $P_k^n(A_n)=\bigl(P_k(\{0\})\bigr)^n=2^{-n}$. Therefore by Lemma~\ref{lem:event-sup} it follows that $\sup_{R\in(\Pcal^n)^{\circ\circ}}R(A_n)=\sup_{P\in\Pcal}P^n(A_n)=2^{-n}$. For any $R\in(\Pcal^n)^{\circ\circ}$, we can write $q:=R(A_n)\le 2^{-n}$ and $p:=Q^n(A_n)=1$ and then apply Lemma~\ref{lem:binary-dpi} to give us that
\[
\KL(Q^n\|R)\ge p\log\frac{p}{q}+(1-p)\log\frac{1-p}{1-q} =\log\frac{1}{q}\ge \log\frac{1}{2^{-n}}=n\log 2.
\]
Therefore, $a_n\ge n\log 2$. Conversely let us choose any $k$ with $R=P_k^n\in\Pcal^n$. Therefore, $a_n\le \KL(Q^n\|P_k^n)=n\KL(Q\|P_k)=n\log 2$. Hence, $a_n=n\log 2$ for every $n$ and thus $\frac1n a_n=\log 2$.
\end{proof}

A natural question is is weak compactness alone enough to give us a non-asymptotic rate identity? Turns out, the answer is no. We will formalize that in the following proposition.

\begin{proposition}\label{prop:weakly-compact-not-enough}
There exists a weakly compact class $\Pcal$ and a probability measure $Q$ such that,
\[
\infKL(Q,\Pcal)<\KLinf(Q,\Pcal).
\]
Therefore, under weak compactness alone the identity $\infKL(Q,\Pcal)=\KLinf(Q,\Pcal)$ is false in general.
\end{proposition}

\begin{proof}
Consider the two-point space $X=\{0,1\}$ with the discrete $\sigma$-algebra. Define $P_1=(1/4, 3/4)$, $P_2=(3/4, 1/4)$, $\Pcal=\{P_1,P_2\}$, and $Q=(1/2,1/2)$. Clearly, since $\Pcal$ is finite, it must be weakly compact. However, it is not convex. We first compute $\infKL(Q,\Pcal)$. As we already noted, we have in general that
\[
\infKL(Q,\Pcal) = R_{\Pcal}(Q) = \sup_{f\in\mathcal B_b(\X,\B)}L(Q,f).
\]
Because we are on a two-point space, every bounded measurable function $f$ is determined by two real numbers, $a=f(0)$ and $b=f(1)$. Therefore it follows that
\begin{align*}
L(Q,f)&=\E_Q[f]-\Lambda_{\Pcal}(f)\\
&=\frac{a+b}{2}-\log\max\left\{\frac14 e^a+\frac34 e^b,\, \frac34 e^a+\frac14 e^b \right\}.
\end{align*}

Let us upper bound this quantity. The maximum of two numbers is at least their average. Hence,
\begin{align*}
\max\left\{\frac14 e^a+\frac34 e^b,\,\frac34 e^a+\frac14 e^b \right\}
&\ge \frac12 \left[\left(\frac14 e^a+\frac34 e^b\right)+\left(\frac34 e^a+\frac14 e^b\right) \right]\\
&=\frac12(e^a+e^b).
\end{align*}

By the arithmetic-geometric mean inequality,
\[
\frac12(e^a+e^b)\ge\sqrt{e^a e^b}=e^{(a+b)/2}.
\]
Therefore, it follows that
\[
\max\left\{ \frac14 e^a+\frac34 e^b,\, \frac34 e^a+\frac14 e^b \right\} \ge e^{(a+b)/2}.
\]
Taking logarithms, it follows that $\Lambda_{\Pcal}(f)\ge \frac{a+b}{2}$. Hence for every $f\in\mathcal B_b(\X,\B)$, $L(Q,f)\le \frac{a+b}{2}-\frac{a+b}{2}=0$. On the other hand, we can easily get equality here by any constant function (i.e.\, $f\equiv c$) because $L(Q,f)=c-\log(e^c)=0$. Thus, $\sup_{f\in\mathcal B_b(\X,\B)}L(Q,f)=0$. Therefore, \underline{$\infKL(Q,\Pcal)=0$}. We now compute the $\KLinf(Q,\Pcal)$. By symmetry,
\[
\KLinf(Q,\Pcal)=\min\{\KL(Q\|P_1),\KL(Q\|P_2)\}=\KL(Q\|P_1).
\]
Let us compute this. Meaning,
\begin{align*}
\KL(Q\|P_1)&=\frac12\log\frac{1/2}{1/4}+ \frac12\log\frac{1/2}{3/4}\\
&=\frac12\log 2+\frac12\log\frac23\\
&=\frac12\log\frac43\\
&>0.
\end{align*}
Therefore, $\KLinf(Q,\mathcal P)>0$. Thus, even though $\Pcal$ was weakly compact, we have $\infKL(Q,\Pcal)=0$ while $\KLinf(Q,\Pcal)>0$. Hence,
\[
\infKL(Q,\Pcal)<\KLinf(Q,\Pcal).
\]
This completes the proof.
\end{proof}

Notice how the class $\mathcal P=\{P_1, P_2\}$ is not convex and $Q$ lies exactly in between the two nulls, $Q=\frac{1}{2}P_1 + \frac{1}{2}P_2$. Essentially, the bipolar enlargement is convexifying the null and therefore pushing the $\infKL$ down to $0$. On the other hand, the $\KLinf$ still is measuring distance to the original nonconvex class $\{P_1,P_2\}$, and this is strictly positive still. Hence the two conditions that together give $\infKL(Q,\Pcal)=\KLinf(Q,\Pcal)$ are weakly compactness and convexity. Weak compactness alone is not enough in this non-asymptotic case.

\section{Composite Alternatives and the Uniform Growth Rate}\label{sec:compQintrins}
In Section~\ref{sec:point-alternative-infKL}, we considered a particular alternative $Q$. We showed that the intrinsic maximal asymptotic expected log growth rate of an $e$-process against a point alternative $Q$ is inherently dependent on the per-sample $\infKL$ rate. In this section, we will broaden this idea to a composite alternative class $\Qcal\subseteq\mathcal M_1(\X)$. A single $e$-process must be able to achieve this growth simultaneously for every $Q\in\Qcal$. Just to be complete here, we will recall all notation. Throughout this section, $(\X,\B)$ is Polish, $\Pcal\subseteq\mathcal M_1(\X)$ and $\Qcal\subseteq\mathcal M_1(\X)$ are nonempty. And, $\Pcal^n:=\{P^n:P\in\Pcal\}$. Recall that $(\Pcal^n)^{\circ\circ}$ is the bipolar associated with nonnegative tests, as in Definition~\ref{def:bipolar}. For each $Q\in\mathcal M_1(\X)$ and $n\in\mathbb N$ we will use the shorthand $a_n(Q):=\infKL(Q^n,\Pcal^n)=\inf_{R\in(\Pcal^n)^{\circ\circ}}\KL(Q^n\|R)\in[0,\infty]$.

The composite case is much more subtle than the singleton alternative $Q$ case. There are at least two composite extensions of $\infKL$ in this case (two natural counterparts of $a_n(Q)$). The first is the worst-case pointwise reverse-projection, which we define as follows.

\begin{definition}\label{def:an-worstcase-pointwise}
For  $n\in\mathbb N$, we define the \underline{worst-case} pointwise reverse-projection value as
\[
a_n(\Qcal):=\inf_{Q\in\Qcal} a_n(Q)=\inf_{Q\in\Qcal}\inf_{R\in\bPn}\KL(Q^n\|R)\in[0,\infty].
\]
We define its corresponding asymptotic rate as
\[
\overline d_{\mathrm{wc}}:=\limsup_{n\to\infty}\frac{1}{n}a_n(\Qcal)\in[0,\infty].
\]
\end{definition}

One can see that $a_n(\Qcal)$ is an extension of the $\infKL$. It is really measuring the smallest reverse-projection value among all the alternatives. We know that every $e$-process needs to work for the worst alternative, hence it follows that $a_n(\Qcal)$ gives us a universal upper bound on what any single $e$-process may guarantee uniformly over the entire $\Qcal$. But, the key point here is that this may not even be achievable uniformly, hence a loose bound. Intuitively, the $e$-variable optimal for one $Q$ may be worse under another alternative $Q'$. To resolve this, it has been shown that the $\inf_{Q\in\Qcal}$ must be inside the actual $e$-power optimization at each horizon \parencite{grunwald2024}. We will formalize that idea of the robust $e$-power at horizon $n$ with the following definition.

\begin{definition}\label{def:bn-robust-epower}
For $n\in\mathbb N$, let us define
\[
b_n(\Qcal,\Pcal):=\sup_{E\in(\Pcal^n)^\circ}\inf_{Q\in\Qcal} \E_{Q^n}[\log E]\in[0,\infty].
\]
The constant $E\equiv 1$ lies in $(\Pcal^n)^{\circ}$, so $b_n(\Qcal,\Pcal)\ge 0$. We define the corresponding asymptotic robust rate as
\[
\overline d_{\mathrm{rob}}:=\limsup_{n\to\infty}\frac{1}{n}b_n(\Qcal,\Pcal)\in[0,\infty].
\]
\end{definition}

One can interpret $b_n(\Qcal,\Pcal)$ as the value in a finite-horizon zero-sum game. The statistician will choose an $e$-variable $E$; nature will choose $Q\in\Qcal$. The payoff is $\E_{Q^n}[\log E]$. However, with the worst-case term $a_n(\Qcal)$, the order of optimization is switched here. That is, we maximize for each $Q$ and then take the worst $Q$. This is the intuition for why we always have the minimax inequality $b_n(\Qcal,\Pcal)\le a_n(\Qcal)$. And this gap is exactly the difficulty in the composite case, which we will attempt to deal with throughout this section. To begin, we will first formalize the idea that the robust $e$-power is dominated by the worst case $\infKL$.

\begin{lemma}\label{lem:bn-le-an}
For every $n\in\mathbb N$, we have that $b_n(\Qcal,\Pcal)\le a_n(\Qcal)$. And, consequently, we have that $\overline d_{\mathrm{rob}}\le \overline d_{\mathrm{wc}}$.    
\end{lemma}

We will formalize what exactly it means for a single $e$-process to grow uniformly over all $Q\in\Qcal$. Specifically, along those deterministic times where the $e$-process is defined, we will measure the uniformity.

\begin{definition}\label{def:uniform-growth-functional}
Each $W\in\Wcal(\Pcal)$ lies in $\mathcal E_{\bbF_{\bt}}$ for some $t\in\Tdet$. As such, we define the corresponding uniform expected-log growth rate for each $W$ as
\[
\mathsf G(W;\Qcal):=
\limsup_{k\to\infty} \frac{1}{t_k}\inf_{Q\in\Qcal}\E_{Q^\infty}[\log W_{t_k}],\text{ for some }\bt\in\Tdet\text{ such that $W\in\mathcal E_{\bbF_{\bt}}$}.
\]
Here, recall that $\bt$ is the unique block schedule associated with $W$. Given this, we define the intrinsic maximal uniform rate as
\[
\Gamma(\Qcal,\Pcal):=\sup_{W\in\Wcal(\Pcal)}\mathsf G(W;\Qcal) \in[0,\infty].
\]
\end{definition}

We are ready to present our main theorem. The theorem can be seen as the analogue of Theorem~\ref{thm:max-eprocess-rate} for composite alternatives. More importantly, our theorem is completely general, as we do not require compactness, domination, semicontinuity, convexity, or any other such assumption. 

\begin{theorem}\label{thm:composite-maxrate}
Suppose that $\Qcal, \Pcal$ are nonempty. Then it follows that
\[
\boxed{
\Gamma(\Qcal,\Pcal)=\overline d_{\mathrm{rob}}= \limsup_{n\to\infty}\frac{1}{n}\Bigl(\sup_{E\in(\Pcal^n)^\circ}\inf_{Q\in\Qcal}\E_{Q^n}[\log E]\Bigr)
}
\]
\end{theorem}

Note that in full generality we will not claim a relationship between $\Gamma(\Qcal,\Pcal)>0$ and the existence of a uniformly power-one sequential test. We formalize a sufficient condition to do so under weak compactness along with a finite $\Qcal$ corollary. 

\begin{remark}\label{rem:compositealternative-eproc-fulltime-sampletime}
By our zeroing argument in Remark~\ref{rem:blockwise-versus-full-point-alternative}, we can get the same value of $\Gamma(\Qcal,\Pcal)$ if the supremum is taken over only $e$-processes on the original filtration $\bbF$. We only use $\Wcal(\Pcal)$ to keep both $e$-processes on $\bbF$ and blockwise $e$-processes on $\bbF_{\bt}$ in one single notation.
\end{remark}

\begin{proof}[Proof of Theorem~\ref{thm:composite-maxrate}]
We will prove this for wealth processes $W\in\Wcal(\Pcal)$. Thanks to Remark~\ref{rem:blockwise-versus-full-point-alternative}, this is equivalent to using only $e$-processes on $\bbF$. We will prove the upper bound $\Gamma(\Qcal,\Pcal)\le \overline d_{\mathrm{rob}}$. We will then show achievability up to some arbitrary $\varepsilon>0$. We will start with the \underline{upper bound}. Consider any $W\in\Wcal(\Pcal)$. If $W\in\mathcal E_{\bbF}(\Pcal)$, we can use the schedule $t_k=k$. And if $W\in\mathcal E_{\bbF_{\bt}}(\Pcal)$ already, we can use its associated schedule $\bt=(t_k)_{k\ge 0}$. Take some particular $k\ge 1$ and set $n=t_k$. $W_{t_k}$ is $\mathcal F_{t_k}$-measurable and $\mathcal F_{t_k}=\sigma(X_1,\dots,X_n)=\pi_n^{-1}(\B^{\otimes n})$ with $\pi_n(\omega):=(\omega_1,\dots,\omega_n)$. Because of this, there exists a $\B^{\otimes n}$-measurable map $w_k:\X^n\to[0,\infty]$ such that for all $\omega\in\X^{\mathbb N}$, $W_{t_k}(\omega)=w_k(\omega_1,\dots,\omega_n)$. Because $\tau\equiv t_k$ is a $(\mathcal F_{t_k})$-stopping rule, Definition~\ref{def:eprocess} gives for every $P\in\Pcal$, $\E_{P^n}[w_k]=\E_{P^\infty}[W_{t_k}]\le 1$. Therefore by definition of the bipolar, $w_k\in(\Pcal^n)^\circ$. And so by Definition~\ref{def:bn-robust-epower} we have
\[
\inf_{Q\in\Qcal}\E_{Q^n}[\log w_k]\le \sup_{E\in(\Pcal^n)^\circ}\inf_{Q\in\Qcal}\E_{Q^n}[\log E]= b_n(\Qcal,\Pcal).
\]
Note that $\E_{Q^\infty}[\log W_{t_k}]=\E_{Q^n}[\log w_k]$. Thus, for each $k$,
\[
\frac{1}{t_k}\inf_{Q\in\Qcal}\E_{Q^\infty}[\log W_{t_k}] \le \frac{1}{n}b_n(\Qcal,\Pcal)=\frac{1}{t_k}b_{t_k}(\Qcal,\Pcal).
\]
Consequently taking the $\limsup_{k\to\infty}$ gives us
\[
\mathsf{G}(W,t;\Qcal)\le \limsup_{k\to\infty}\frac{1}{t_k}b_{t_k}(\Qcal,\Pcal)\le \limsup_{n\to\infty}\frac{1}{n}b_n(\Qcal,\Pcal) = \overline d_{\mathrm{rob}}.
\]
Taking the supremum over all $e$-processes gives $\Gamma(\Qcal,\Pcal)\le \overline d_{\mathrm{rob}}$. It remains to now prove the \underline{lower bound}; the achievability. We first consider the case where $\overline d_{\mathrm{rob}}=+\infty$. In this case, take some arbitrary $M>0$. By definition of $\limsup$, there exists some $m\in\mathbb N$ with $b_m(\Qcal,\Pcal)\ge mM+1$ that is in $[0,\infty]$. By definition of $b_m(\Qcal,\Pcal)$, there exists some $E^\star\in(\Pcal^m)^\circ$ such that $\inf_{Q\in\Qcal}\E_{Q^m}[\log E^\star]\ge mM$. Hence, let us now repeat this same block construction that we had in Proposition~\ref{prop:achieve-infkl-rate}. Only this time, we will use $E^\star$. Let us also define $t_k:=km$ and
\[
W_{t_k}:=\prod_{j=1}^k E^\star\bigl(X_{(j-1)m+1},\dots,X_{jm}\bigr).
\]
Note that $E^\star\in(\Pcal^m)^\circ$, so we can use our same conditional expectation calculation that we did in Proposition~\ref{prop:achieve-infkl-rate}. If we do this, we can clearly see that $W_{t_k}$ is a $e$-process for $\Pcal$. And moreover, for any $Q\in\Qcal$ because we have independence of blocks under $Q^\infty$ we have that
\[
\E_{Q^\infty}[\log W_{t_k}]=k\E_{Q^m}[\log E^\star]\ge kmM = t_k M.
\]
Therefore, $\mathsf{G}(W,t;\Qcal)\ge M$. And since $M>0$ was arbitrary it follows that $\Gamma(\Qcal,\Pcal)=+\infty=\overline d_{\mathrm{rob}}$ in this case. We will now consider the other case, where $\overline d_{\mathrm{rob}}<\infty$. Take some $\varepsilon>0$. By definition of $\limsup$ we know that there exists some $m\in\mathbb N$ such that
\begin{equation}\label{eq:choose-m-robust}
\frac{1}{m}b_m(\Qcal,\Pcal)\ge \overline d_{\mathrm{rob}}-\frac{\varepsilon}{2}.    
\end{equation}
By definition of $b_m(\Qcal,\Pcal)$, there exists some $E^\star\in(\Pcal^m)^\circ$ such that
\begin{equation}\label{eq:near-opt-robust}
\inf_{Q\in\Qcal}\E_{Q^m}[\log E^\star]\ge b_m(\Qcal,\Pcal)-\frac{\varepsilon m}{2}.
\end{equation}

Let us now build the same product test supermartingale with $t_k=km$ and $W_{t_k}=\prod_{j\le k}E^\star(\text{block }j)$. Then, for each of the $Q\in\Qcal$ we have that
\[
\E_{Q^\infty}[\log W_{t_k}]=k\E_{Q^m}[\log E^\star]\ge k\inf_{Q'\in\Qcal}\E_{Q'^m}[\log E^\star].
\]

If we just take the infimum over $Q\in\Qcal$ on the left hand side and use \eqref{eq:near-opt-robust}, we get
\[
\inf_{Q\in\Qcal}\E_{Q^\infty}[\log W_{t_k}]\ge k\Bigl(b_m(\Qcal,\Pcal)-\frac{\varepsilon m}{2}\Bigr).
\]
Then dividing by $t_k=km$ gives us that
\[
\frac{1}{t_k}\inf_{Q\in\Qcal}\E_{Q^\infty}[\log W_{t_k}] \ge \frac{1}{m}b_m(\Qcal,\Pcal)-\frac{\varepsilon}{2}.
\]

Now, if we take the $\limsup_{k\to\infty}$, the rhs is unchanged. So doing this followed by applying \eqref{eq:choose-m-robust} gives us that $\mathsf{G}(W,t;\Qcal)\ge \overline d_{\mathrm{rob}}-\varepsilon$. Recall $\varepsilon>0$ was arbitrary. Hence, $\Gamma(\Qcal, \Pcal)\ge \overline d_{\mathrm{rob}}$. Together with the upper bound, this tells us that $\Gamma(\Qcal,\Pcal)=\overline d_{\mathrm{rob}}$. And, the inequality $\Gamma(\Qcal,\Pcal)\le \overline d_{\mathrm{wc}}$ is immediate from Lemma~\ref{lem:bn-le-an}. Finally, thanks to Remark~\ref{rem:compositealternative-eproc-fulltime-sampletime}, the supremum will not change if we take it over all $e$-processes rather than the blockwise ones.
\end{proof}

Having shown this, let us explain a little bit on the almost sure growth under each $Q$. In this above achievability construction, take some $m$ and $E^\star$. For each $Q\in\Qcal$ where $\E_{Q^m}[|\log E^\star|]<\infty$, the block increments $\log E^\star(\text{block }j)$ are i.i.d.\ integrable under $Q^\infty$. Therefore, by the strong law we have that with probability 1 under $Q^\infty$,
\[
\frac{1}{t_k}\log W_{t_k} \longrightarrow \frac{1}{m}\E_{Q^m}[\log E^\star].
\]
Thus, whenever we have integrability, the achieved uniform expected-log rate is also realized almost surely under $Q$ along the block times.

\subsection{When do the robust and worst-case pointwise rates actually match?}
Theorem~\ref{thm:composite-maxrate} completely characterizes the maximal expected log growth rate in terms of the robust finite-horizon values $b_n(\Qcal,\Pcal)$. That being said, this $b_n(\Qcal,\Pcal)$ can indeed be arduous to compute. A more explicit and intuitive quantity is the worst-case pointwise reverse-projection $a_n(\Qcal)$ from Definition~\ref{def:an-worstcase-pointwise}. In general, this will only give us an upper bound. Meaning, by Lemma~\ref{lem:bn-le-an}, we have that $b_n\le a_n(\Qcal)$ always. This gap, therefore, is a serious composite-alternative mini-max gap. However, we will give broad conditions under which said gap is going to vanish at the asymptotic per-sample rate level. To begin, we will start with the finite-class case. That is, for finite $\Qcal$, this asymptotic mini-max gap is negligible.

\begin{proposition}\label{prop:finite-Q-gap}
Assume that $\Qcal=\{Q_1,\dots,Q_M\}$ is finite. Then, for every $n\in\mathbb N$ and every $\eta>0$ that, $b_n(\Qcal,\Pcal)\ge a_n(\Qcal)-\log M-\eta$. Consequently,
\[
\limsup_{n\to\infty}\frac{1}{n}b_n(\Qcal,\Pcal) = \limsup_{n\to\infty}\frac{1}{n}a_n(\Qcal).
\]
\end{proposition}

\begin{proof}
Take some particular but arbitrarily chosen $n\in\mathbb N$ throughout this proof. The first case we will consider is the case where \underline{$a_n(\Qcal)=+\infty$}. Clearly, then $a_n(Q_i)=+\infty$ for every $i\in\{1,\dots, M\}$. Now, let us fix some $L>0$. For each $i$, since $\sup_{E\in(\Pcal^n)^\circ}\E_{Q_i^n}[\log E]=+\infty$, we may choose some $E_i\in(\Pcal^n)^\circ$ such that
\[
\E_{Q_i^n}[\log E_i]\ge L+\log M.
\]
Now let $E:=\frac1M\sum_{i=1}^M E_i$. By linearity of expectation under each $P^n$, it follows that $E\in(\Pcal^n)^\circ$. In addition, we have that $E\ge \frac1M E_i$ pointwise. So $\log E\ge \log E_i-\log M$ pointwise. Therefore, for each $i$ we get that
\[
\E_{Q_i^n}[\log E] \ge \E_{Q_i^n}[\log E_i]-\log M \ge L.
\]
So, $\inf_{Q\in\Qcal}\E_{Q^n}[\log E]\ge L$. Hence, $b_n(\Qcal,\Pcal)\ge L$. But, $L$ was arbitrary, which means that $b_n(\Qcal,\Pcal)=+\infty$, and so therefore the claim inequality is trivial. It remains to now prove the case where \underline{$a_n(\Qcal)<\infty$}. Let $d:=a_n(\Qcal)=\min_{1\le i\le M}a_n(Q_i)$ and take some particular $\eta>0$. For each $i$, $a_n(Q_i)\ge d$. Therefore by the definition of the supremum, there must exist $E_i\in(\Pcal^n)^\circ$ such that
\[
\E_{Q_i^n}[\log E_i]\ge d-\eta.
\]
Now define $E:=\frac1M\sum_{i=1}^M E_i\in(\Pcal^n)^\circ$. As we already argued in the first case, we have that $\log E\ge \log E_i-\log M$ pointwise. Hence for each $i$ that we have that
\[
\E_{Q_i^n}[\log E]\ge \E_{Q_i^n}[\log E_i]-\log M\ge d-\eta-\log M.
\]
Now if we take the $\inf_{i\le M}$ we get that $\inf_{Q\in\Qcal}\E_{Q^n}[\log E]\ge a_n(\Qcal)-\log M-\eta$. Then, finally to conclude we can just take the $\sup_{E\in(\Pcal^n)^\circ}$ over the left hand side, giving $b_n(\Qcal,\Pcal)\ge a_n(\Qcal)-\log M-\eta$.
\end{proof}

Notice how in this above proposition we have shown that not knowing which $Q\in\Qcal$ holds creates a cost at most an additive $\log M$ in the finite horizon. And, this disappears in such per-sample asymptotics. However, the question becomes what about in the case of possibly uncountable $\Qcal$. That will be the motivation for our following theorem under which a sub-exponential complexity covering condition makes the same conclusion hold. In other words, if we have subexponential uniform approximability, that will imply that $b_n$ and $a_n(\Qcal)$ will match asymptotically.

\begin{theorem}\label{thm:subexp-cover}
Assume that there exist sequences $(\Delta_n)_{n\ge 1}\subset[0,\infty)$ and integers $(N_n)_{n\ge 1}$ such that the following hold.
\begin{enumerate}
    \item For each $n$, there exist $E_{n,1},\dots,E_{n,N_n}\in(\Pcal^n)^\circ$ with the uniform approximation property; meaning, that,
    \[
    \forall Q\in\Qcal,\quad\exists j\in\{1,\dots,N_n\}\quad\text{s.t.}\quad \E_{Q^n}[\log E_{n,j}] \ge a_n(Q)-\Delta_n.
    \]
    \item Second, the approximation error and complexity are sublinear. That is,
    \[
    \frac{\Delta_n}{n}\longrightarrow 0 \qquad\text{and}\qquad \frac{\log N_n}{n}\longrightarrow 0.
    \]
\end{enumerate}
Then it once again follows that,
\[
\limsup_{n\to\infty}\frac{1}{n}b_n(\Qcal,\Pcal) = \limsup_{n\to\infty}\frac{1}{n}a_n(\Qcal)
\]
\end{theorem}

\begin{proof}
Consider some $n\in\mathbb N$ and defining the uniform mixture, $E_n:=\frac{1}{N_n}\sum_{j=1}^{N_n} E_{n,j}$. The same linearity argument that we did in the Proposition~\ref{prop:finite-Q-gap} gives $E_n\in(\Pcal^n)^\circ$. Now, consider any $Q\in\Qcal$ and choose an index $j(Q)\in\{1,\dots,N_n\}$ which satisfies the approximation property. Pointwise, we have $E_n\ge \frac{1}{N_n}E_{n,j(Q)}$. Therefore, $\log E_n\ge \log E_{n,j(Q)}-\log N_n$. Hence,
\[
\E_{Q^n}[\log E_n]\ge \E_{Q^n}[\log E_{n,j(Q)}]-\log N_n\ge a_n(Q)-\Delta_n-\log N_n.
\]
Taking the $\inf_{Q\in\Qcal}$, we get
\[
\inf_{Q\in\Qcal}\E_{Q^n}[\log E_n] \ge \inf_{Q\in\Qcal}a_n(Q)-\Delta_n-\log N_n= a_n(\Qcal)-\Delta_n-\log N_n.
\]
Necessarily $E_n\in(\Pcal^n)^\circ$. So the definition of $b_n(\Qcal,\Pcal)$ tells us that
\[
b_n(\Qcal,\Pcal) \ge a_n(\Qcal)-\Delta_n-\log N_n.
\]

All we need to do now is take this result together with the universal upper bound, that $b_n(\Qcal,\Pcal)\le a_n(\Qcal)$ from Lemma~\ref{lem:bn-le-an}. Let us do this, divide by $n$ and take the $\limsup_{n\to\infty}$. By assumption, we have that $\Delta_n=o(n)$ and $\log N_n=o(n)$. This thereby forces the lower bound to match the upper bound in the limsup, which proves our claim.
\end{proof}

Theorem~\ref{thm:subexp-cover} has a small requirement that at a horizon $n$, we should be able to find a subexponential family of $e$-variables which have expected log-payoff that approximates the individually optimal payoff $a_n(Q)$ up to $o(n)$, but uniformly over all the $Q\in\Qcal$. This is the smallest way we can give substance to the idea that $\Qcal$ not be too large from the perspective of uniform log-growth. If $N_n$ is exponentially large, then the mixture penalty $\log N_n$ will be linear in $n$. As a consequence, we will have a genuine rate gap. However, we once again emphasize that we can analyze the intrinsic uniform rate without any assumptions through the robust horizon values $b_n(\Qcal,\Pcal)$.

\subsection{When will the Intrinsic Uniform Growth Rate be Positive?}
Let us state what we can get from the fact that these pointwise $\KL$ quantities are positive. Before getting there, note that $\Gamma(\Qcal,\Pcal)$ depends on the robust finite-horizon values $b_n(\Qcal,\Pcal)$. Throughout this section, note that all statements also apply for $\bbF$-$e$-processes. On the other hand, the items we will expound on in this section will be those build from the pointwise values,
\[
a_n(Q):=\inf_{R\in(\Pcal^n)^{\circ\circ}}\KL(Q^n\|R),
\]
where $Q\in\mathcal M_1(\X)$. In full generality we can only get necessary and sufficient statements with additional assumptions. However, we will give some necessary implications also, which hold in full generality. For each $Q\in\mathcal M_1(\X)$, recall that we showed the sequence $(a_n(Q))_{n\ge 1}$ is superadditive. So, as in Lemma~\ref{lem:an-superadditive-sec}, we may define the asymptotic pointwise rate as
\[
d_{\star}(Q,\Pcal):=\lim_{n\to\infty}\frac1 n a_n(Q)= \sup_{n\ge 1}\frac1n a_n(Q)\in[0,\infty],
\]
where the limit exists by Lemma~\ref{lem:an-superadditive-sec}. Therefore, in full generality, it will always hold that for any arbitrary $\Qcal,\Pcal$,
\begin{equation}\label{eq:eprocpos-nec}
\boxed{
\Gamma(\Qcal,\Pcal)>0\Longrightarrow  \inf_{Q\in\Qcal} \lim_{n\to\infty}\frac1n \KLinf(Q^n,(\Pcal^n)^{\circ\circ}) > 0 \Longrightarrow  \inf_{Q\in\Qcal} \KLinf(Q,\Pcal) > 0.
}
\end{equation}
We will show that this holds in Theorem~\ref{thm:gamma-positive-necessary}. The converses are in general not true without additional structure. However, as an example, suppose that $\Qcal$ is finite and that $\Phi(R)$ is weakly lower semicontinuous at $Q$. Then, we will have that (which will be proven in Corollary~\ref{cor:gamma-positive-finiteQ-wlsc}),
\begin{equation}\label{eq:necsufeprocpos}
\boxed{
\Gamma(\Qcal,\Pcal)>0\Longleftrightarrow  \inf_{Q\in\Qcal} \lim_{n\to\infty}\frac1n \KLinf(Q^n,(\Pcal^n)^{\circ\circ}) > 0 \Longleftrightarrow  \inf_{Q\in\Qcal} \KLinf(Q,\Pcal) > 0.
}
\end{equation}
We are now ready to present our theorem on necessary positivity implications. These will be fully general and consequently prove \eqref{eq:eprocpos-nec}.

\begin{theorem}\label{thm:gamma-positive-necessary}
Suppose $\Qcal,\Pcal$ are nonempty. Then, $\Gamma(\Qcal,\Pcal)\le \inf_{Q\in\Qcal}d_{\star}(Q,\Pcal)$. This means that if $\Gamma(\Qcal,\Pcal)>0$, then $\inf_{Q\in\Qcal}d_{\star}(Q,\Pcal)>0$.  In addition, for every $Q\in\mathcal M_1(\X)$, $d_{\star}(Q,\Pcal)\le \KLinf(Q,\Pcal)$. Hence, $\Gamma(\Qcal,\Pcal)>0$ also implies that $\inf_{Q\in\Qcal}\KLinf(Q,\Pcal)>0$.
\end{theorem}

\begin{proof}
Consider any $e$-process $(W_{t_k})_{k\ge 0}$ for $\Pcal$ along deterministic sample sizes $t_k$. Take some $Q\in\Qcal$. For each $k$, set $n:=t_k$. Since $W_{t_k}$ is $\mathcal F_{t_k}$-measurable, there exists a $B^{\otimes n}$-measurable function $w_k:\X^n\mapsto[0,\infty]$ such that $W_{t_k}(\omega)=w_k(\omega_1,\dots,\omega_n)$. Each deterministic $t_k$ is a stopping time, hence by the definition of $e$-process, for every $P\in\Pcal$, $\E_{P^\infty}[W_{t_k}]\le 1$. However, since $W_{t_k}=w_k(X_1,\dots,X_n)$, it follows that this is equivalently $\E_{P^n}[w_k]\le 1$ for all $P\in\Pcal$, so we must have $w_k\in(\Pcal^n)^{\circ}$. Therefore, thanks to Proposition~\ref{prop:infkl-e-power} at the horizon $n=t_k$ we get
\begin{equation}\label{eq:Wk-upper-akQ}
\E_{Q^\infty}[\log W_{t_k}]=\E_{Q^n}[\log w_k]\le \sup_{E\in(\Pcal^n)^{\circ}}\E_{Q^n}[\log E]=a_n(Q)=a_{t_k}(Q).    
\end{equation}
Therefore it follows that
\begin{align}
\mathsf G(W,t;Q)&=\limsup_{k\to\infty}\frac{1}{t_k}\inf_{Q'\in\Qcal}\E_{Q'^{\infty}}[\log W_{t_k}]\notag\\
&\le \limsup_{k\to\infty}\frac{1}{t_k}\E_{Q^\infty}[\log W_{t_k}]\label{eq:G-upper1}\\
&\le \limsup_{k\to\infty}\frac{1}{t_k}a_{t_k}(Q)\label{eq:G-upper2}\\
&\le \lim_{n\to\infty}\frac1n a_n(Q)\label{eq:G-upper3}\\
&=d_{\star}(Q,\Pcal)\notag.
\end{align}
Here, \eqref{eq:G-upper1} is obvious. From \eqref{eq:Wk-upper-akQ}, \eqref{eq:G-upper2} is immediate. Lastly, \eqref{eq:G-upper3} holds because a limsup along some subsequence $(t_k)$ is bounded by the full limit of $a_n(Q)/n$. Now this bounds holds for any $Q\in\Qcal$, so it follows that $\mathsf{G}(W,t;\Qcal)\le \inf_{Q\in\Qcal}d_{\star}(Q,\Pcal)$. Taking the supremum over all $e$-processes $W_{t_k}$ gives us that
\[
\Gamma(\Qcal,\Pcal)\le \inf_{Q\in\Qcal}d_{\star}(Q,\Pcal).
\]
So our first claim is proved. It remains for us to show that for each $Q$, $d_{\star}(Q,\Pcal)\le \KLinf(Q,\Pcal)$. For every $n\in\mathbb N$, we have that
\[
a_n(Q)=\inf_{R\in(\Pcal^n)^{\circ\circ}}\KL(Q^n\|R)\le \inf_{P\in\Pcal}\KL(Q^n\|P^n)=\inf_{P\in\Pcal}n\KL(Q\|P)=n\KLinf(Q,\Pcal).
\]
Dividing this by $n$ and letting $n\to\infty$ gives $d_{\star}(Q,\Pcal)\le \KLinf(Q,\Pcal)$. Taking the infimum over all $Q\in\Qcal$ completes the proof.
\end{proof}

As we mentioned, to get sufficient conditions, there needs to be more structure. One immediate one is if we suppose that there is no asymptotic minimax gap. Because the robust rate is the best achievable growth rate of the $e$-process, implying that this growth rate will also be equal to worst-case rate if there is no asymptotic gap. Then, one-step bipolar positivity becomes sufficient. We will formalize that idea as follows. Recall that $a_n(\Qcal):=\inf_{Q\in\Qcal}a_n(Q)$.

\begin{corollary}\label{cor:bipolar-one-step-positive}
Assume in addition that
\[
\Gamma(\Qcal,\Pcal)=\overline d_{\mathrm{wc}}:=\limsup_{n\to\infty}\frac1n a_n(\Qcal).
\]
This will hold for example, under Proposition~\ref{prop:finite-Q-gap}. And it will also hold under any assumption that gives $\overline d_{\mathrm{rob}}=\overline d_{\mathrm{wc}}$. 
Then it follows that $\inf_{Q\in\Qcal}\KLinf(Q,\Pcal^{{\circ}{\circ}})>0$ implies that $\Gamma(\Qcal,\Pcal)>0$. And in particular, all the single directions in \eqref{eq:eprocpos-nec} become double directions, i.e.\ both necessary and sufficient as in \eqref{eq:necsufeprocpos}.
\end{corollary}

We will now specialize. First, we will give a condition that gives an if and only if statement. It is for finite $\Qcal$.

\begin{corollary}\label{cor:gamma-positive-finiteQ}
Assume that $\Qcal=\{Q_1,Q_2,...,Q_M\}$ is finite. Then it follows that
\[
\Gamma(\Qcal,\Pcal)=\inf_{Q\in\Qcal}d_{\star}(Q,\Pcal)=\inf_{Q\in\Qcal}\lim_{n\to\infty}\frac1n \KLinf(Q^n, (\Pcal^n)^{{\circ}{\circ}}).
\]
Therefore $\Gamma(\Qcal,\Pcal)>0$ implies that
\[
\inf_{Q\in\Qcal}\lim_{n\to\infty}\frac1n \KLinf(Q^n,(\Pcal^n)^{{\circ}{\circ}})>0.
\]
\end{corollary}

Lastly, to get a if and only if condition with respect to the $\KLinf$, we will need both finite $\Qcal$ and pointwise weak lower semicontinuity. We will formalize that as follows.

\begin{corollary}\label{cor:gamma-positive-finiteQ-wlsc}
Assume that $\Qcal=\{Q_1,\dots,Q_M\}$ is finite and that for each $Q\in\Qcal$ the map $\Phi(R):=\inf_{P\in\Pcal}\KL(R\|P)$ is weakly lower semicontinuous at $Q$. Then, for every $Q\in\Qcal$, $d_{\star}(Q,\Pcal)=\KLinf(Q,\Pcal)$. Hence, $\Gamma(\Qcal,\Pcal)=\inf_{Q\in\Qcal}\KLinf(Q,\Pcal)$. Therefore
\[
\boxed{
\Gamma(\Qcal,\Pcal)>0\Longleftrightarrow \inf_{Q\in\Qcal}d^{\star}(Q,\Pcal)>0 \Longleftrightarrow \inf_{Q\in\Qcal}\KLinf(Q,\Pcal)>0.
}
\]
\end{corollary}

We would like to make an important note on the testability problem now. Thanks to Theorem~\ref{thm:powerone-characterization}, we always have a necessary and sufficient condition for characterizing the existence of power one sequential tests for arbitrary composite nulls against point alternatives. However, one question remains here: do we have such a statement for arbitrary composite alternatives? Is $\Gamma(\Qcal,\Pcal)>0$ both necessary and sufficient for the existence of a level-$\alpha$ power one sequential test of $\Pcal$ against $\Qcal$? The answer turns out to be no. We will convince ourselves of this with a very simple (countable) counterexample.

\begin{example}\label{ex:baby-breakiffseqbip}
Consider the null $\Pcal=\{P_0\}$, where $P_0\sim\mathrm{Bern}(1/2)$. Now, let $\Qcal:=\{Q_i: i\ge 1\}$, where $Q_i\sim\mathrm{Bern}(1/2 + 2^{-i-2})$. Here, each $Q_i$ will be individually power-one against $\{P_0\}$ thanks to Theorem~\ref{thm:powerone-characterization}, since $\KL(Q_i\|P_0)>0$. Now let us choose level allocations $\alpha_i:=\alpha 2^{-i}$. And let $\tau_i$ be a level-$\alpha_i$ power one test against $Q_i$. Define the stopping rule $\tau:=\inf_{i\ge 1}\tau_i$. It follows then that
\[
P_0^{\infty}(\tau<\infty)\le \sum_{i\ge 1} P_0^{\infty}(\tau_i<\infty)\le \sum_{i\ge1}\alpha_i=\alpha.
\]
However, for each $i$, $Q_i^{\infty}(\tau<\infty)=1$ since $\tau<\tau_i$. Therefore, $\Qcal$ is clearly power-one sequentially testable against $\Pcal$ in this case. But recognize that since $\KL(Q_i\|P_0)\downarrow 0$ we have that, thanks to Theorem~\ref{thm:gamma-positive-necessary},
\[
\Gamma(\Qcal,\Pcal)\le \inf_{Q\in\Qcal} d_{\star} (Q,\Pcal)\le \inf_{i\ge 1}\KL(Q_i\|P_0)=0.
\]
Thus, we have a power one test for arbitrary $\Qcal,\Pcal$ albeit $\Gamma=0$, proving that $\Gamma(\Qcal,\Pcal)$ being strictly positive is not necessary to have a power one test exist. 
\end{example}

\subsection{Pointwise Power-One Test Characterization for composite $\Pcal$ against composite $\Qcal$}

The most fully general characterization of the optimal uniform growth rate is already presented in Theorem~\ref{thm:composite-maxrate}. However, the case that is most easy to see is the weakly compact one. In this regime, we already have a uniform positive lower bound on the $\KLinf(Q,\Pcal)$ over all $Q\in\Qcal$. From it, we have a single test supermartingale that works simultaneously for all alternatives. To be clear, we say that a level-$\alpha$ sequential test is pointwise power one against $\Qcal$ if $\sup_{P\in\Pcal}P^{\infty}(\tau<\infty)\le \alpha$ and for every $Q\in\Qcal$, $Q^{\infty}(\tau<\infty)=1$.

\begin{theorem}\label{thm:uniform-power-one-compact}
Suppose that $\X$ is Polish. Suppose that $\Pcal,\Qcal\subseteq \mathcal M_1(\X)$ are weakly compact. And suppose that $\Qcal\subseteq\Pcal^c$. Recall that $d_{\min}:=\inf_{Q\in\Qcal}\KLinf(Q,\Pcal)$. Suppose that $d_{\mathrm{min}}>0$. Then it follows that for every $\varepsilon\in (0,d_{\mathrm{min}})$, there exist deterministic times $t_k\uparrow \infty$ and a single test supermartingale $(E_{t_k})_{k\ge 0}$ for $\Pcal$ such that for every $Q\in\Qcal$ wp 1 under $Q^\infty$,
\[
\liminf_{k\to\infty}\frac{1}{t_k}\log E_{t_k} \ge d_{\min}-\varepsilon.
\]
And,
\[
\liminf_{k\to\infty}\frac{1}{t_k}\inf_{Q\in\Qcal}\E_{Q^\infty}[\log E_{t_k}] \ge d_{\min}-\varepsilon.
\]
As a consequence, for every $\alpha\in(0,1)$, the stopping time $\tau_{\alpha}:=\inf\{{t_k}\ge 1: E_{t_k}\ge 1/\alpha\}$ is a power-one level-$\alpha$ test of $\Pcal$ against $\Qcal$.
\end{theorem}

The takeaway here is that weak compactness allows us the lower semicontinuity, which in turn will control the $\KL$ uniformly over $\Qcal$. Our assumption that $d_{\mathrm{min}}>0$ entails that the entire alternative class is uniformly separated from the null. Hence we get a single test supermartingale that grows at a strictly positive asymptotic rate no matter which $Q\in\Qcal$ is true. And if we threshold that process, we will get a sequential test that is pointwise power-one.

Now, we will upgrade our earlier Corollary~\ref{cor:gamma-positive-finiteQ-wlsc} to a more interpretable one that also gives us a uniform power one test.

\begin{corollary}\label{cor:uniform-power-one-finite}
Suppose that $\X$ is Polish. And suppose that $\Pcal\subseteq \mathcal M_1(\X)$ is weakly compact. Assume that $\Qcal=\{Q_1,\dots,Q_M\}$ is finite. Then, the following conditions are equivalent:
\begin{enumerate}
    \item $\Gamma(\Qcal,\Pcal)>0$.
    \item $\inf_{Q\in\Qcal}\KLinf(Q,\Pcal)>0$.
\end{enumerate}
And if these equivalent conditions hold, for every $\alpha\in(0,1)$, there exists a power-one level-$\alpha$ test of $\Pcal$ against $\Qcal$.
\end{corollary}

Now, we will make this more concrete by upgrading this corollary. We will present two general sufficient conditions for power-one sequential tests against $\Pcal$ and $\Qcal$. 

\begin{theorem}\label{thm:two-sufficient-powerone}
Let $\Pcal,\Qcal\subseteq\mathcal M_1(\X)$ be nonempty and nonintersecting. Suppose that either of the following two conditions hold:
\begin{enumerate}
    \item $\Gamma(\Qcal,\Pcal)>0$.
    \item $\Pcal$ is weakly compact.
\end{enumerate}
Then, for every $\alpha\in(0,1)$, there exists a level $\alpha$ pointwise power-one sequential test of $\Pcal$ against $\Qcal$.
\end{theorem}

Notice that throughout for this composite alternative case, we have not presented any condition that is both necessary and sufficient. Now, if a level $\alpha$ pointwise power one test of $\Pcal$ against $\Qcal$ were to exist, thanks to Theorem~\ref{thm:powerone-characterization} applied pointwise to $Q\in\Qcal$,
\[
d_\star(Q,\Pcal)=\lim_{n\to\infty}\frac1n\inf_{R\in(\Pcal)^{\circ\circ}}\KL(Q^n\|R)>0.
\]
In other words, this condition is necessary (pointwise). However, it is not sufficient in general for a single test to work against all $Q\in\Qcal$ without some additional structure.

\section{Why (Blockwise) Test Supermartingales Suffice}\label{sec:super}

We will now present several structural theorems for rate optimality. The first will be existence of test supermartingale achieving the $\KLinf$ growth rate under weak lower semicontinuity. Let us now build a test-supermartingale as in Definition~\ref{def:test-supermartingale} for the composite null $\Pcal$ whose asymptotic logarithmic growth rate under $Q$ is in fact $\KLinf(Q,\Pcal)$. In this construction specifically we are going to use geometrically increasing blocks and blockwise typical set e variables. In building this, importantly, note this will be under no assumptions whatsoever on the point alternative $Q$.

\subsection{Test supermartingale that achieves asymptotic growth rate $\KLinf(Q,\Pcal)$}

\begin{proposition}\label{prop:eprocess-rate}
Let $Q\in\mathcal M_1(\X)$ and let $\Pcal\subseteq\mathcal M_1(\X)$ be nonempty. Assume that $\KLinf(Q,\Pcal)<\infty$ and that $\Pcal$ is $\KLinf$ lower semicontinuous at $Q$. Let $d:=\KLinf(Q,\Pcal)=\inf_{P\in\Pcal}\KL(Q\|P)\in[0,\infty)$. Then, it follows that there exists a test supermartingale $(E_{t_k})_{k\ge 0}$ for $\Pcal$ and an associated nondecreasing sequences of sample sizes $(t_k)_{k\ge 0}$ with $t_0=0$ and $t_k\to\infty$ such that under $Q^\infty$,
\[
\lim_{k\to\infty}\frac{1}{t_k}\log E_{t_k} = d\quad{\text{almost surely.}}
\]
And also,
\[
\lim_{k\to\infty}\frac{1}{t_k}\E_{Q^\infty}[\log E_{t_k}]=d.
\]
If $d=0$ then we can take $E_{t_k}\equiv1$ for all $k$.
\end{proposition}

We will use the following entropy inequality in the proof.

\begin{lemma}\label{lem:entropy-ineq}
Let $(\Omega, \mathcal F)$ be measurable and $\mu, \nu\in\mathcal M_1(\Omega)$. Now, then it follows that for any measurable $f$ with $\int e^{f}d\nu<\infty$ we have that $\int fd\mu \le \KL(\mu\|\nu)+\log\int e^{f}d\nu$.    
\end{lemma}

In addition, we will also need a lemma for illustrating that the empirical measures converge under $Q$. 

\begin{lemma}\label{lem:empirical-conv}
Suppose that $Q\in\mathcal{M}_1(\X)$ and that $(X_i)_{i\ge 1}$ are i.i.d.\ with law $Q$. Now let $\widehat{Q}_n:=\frac{1}{n}\sum_{i=1}^n\delta_{X_i}$ be the empirical measure. Then it follows that $\widehat{Q}_n\Rightarrow Q$ almost surely. And thus consequently for every weakly open neighborhood $U$ of $Q$ we have that
\[
Q^n\bigl(\widehat{Q}_n\in U\bigr)\longrightarrow1.
\]
\end{lemma}

Note that it's the case that the empirical measure map $x^n\mapsto \widehat Q_n(x^n)$ from $(\X^n,\B^{\otimes n})$ into $\mathcal{M}_1(\X)$ (endowed with the Borel $\sigma$-field of the weak topology) is measurable. As a consequence, sets of the form $\{\widehat Q_n\in U\}$ are $B^{\otimes n}$-measurable whenever $U$ is weakly Borel, in particular whenever $U$ is weakly open.

\begin{proof}[Proof of Proposition~\ref{prop:eprocess-rate}]
If $d=0$, then it follows that we can take $E_{t_k}\equiv 1$ and any $t_k\to\infty$, hence the conclusion of our proposition is immediate. Hence going forward, wlog we can assume that $d>0$. To this end, let us first begin by letting $(\varepsilon_k)_{k\ge 1}$ be a positive sequence such that $\varepsilon_k\downarrow 0$ and
\[
\frac{\sum_{j=1}^k\varepsilon_j m_j}{\sum_{j=1}^k m_j}\longrightarrow 0,
\]
for any sequence $(m_j)$ with $m_{j+1}\ge 2m_j$. To be concrete, one can take $\varepsilon_k:=\min\{2^{-k}, d/4\}$ so that $\varepsilon_k\downarrow 0$ and $d-2\varepsilon_k\ge d/2>0$ for all $k$. Then, let us define $\Phi(R):=\inf_{P\in\Pcal}\KL(R\|P)$. By $\KLinf$ lower semicontinuity of $\Phi$ at $Q$ that for each $k\ge 1$, the set $O_k:=\{R\in\mathcal M_1(\X):\Phi(R)>d-\varepsilon_k\}$ is a weakly open neighborhood of $Q$. Let us choose $\delta_k>0$ such that the closed and bounded Lipschitz ball
\[
C_k:=\overline B^{\mathrm{BL}}_{\delta_k}(Q)=\{R\in\mathcal M_1(\X):d_{\mathrm{BL}}(R,Q)\le\delta_k\}
\]
satisfies $C_k\subseteq O_k$. It follows then that $C_k$ is weakly closed and convex, and we have 
\[
\inf_{R\in C_k}\Phi(R)\ge d-\varepsilon_k.
\]
Since, $d_{\mathrm{BL}}(\widehat Q_m, Q)\to 0$ almost surely under $Q^{\infty}$, as $m\to\infty$, we have that $Q^m(\widehat Q_m\in C_k)\to 1$ for each $k$. Thanks to Lemma~\ref{lem:convex-empirical}, then
\[
\sup_{P\in\Pcal} P^m(\widehat Q_m\in C_k) \le \exp\Bigl(-m\inf_{R\in C_k}\Phi(R)\Bigr) \le \exp\bigl(-m(d-\varepsilon_k)\bigr)
\]
Let us now choose the block lengths $m_k$ recursively so that $m_k\ge 2m_{k-1}$ for $k\ge 2$ and
\begin{equation}\label{eq:eproc-Qgood}
Q^{m_k}\bigl(\widehat Q_{m_k}\in C_k\bigr)\ge 1-2^{-k}    
\end{equation}
Now define $A_k:=\{\widehat Q_{m_k}\in C_k\}\subseteq X^{m_k}$. Then
\begin{equation}\label{eq:eproc-betabound}
\sup_{P\in\Pcal}P^{m_k}(A_k) \le \exp\bigl(-m_k(d-\varepsilon_k)\bigr) \le \exp\bigl(-m_k(d-2\varepsilon_k)\bigr)   
\end{equation}
Now, let us define the cumulative sample sizes $t_0:=0$ and $t_k:=\sum_{j=1}^k m_j$. Now let $(\gamma_k)_{k\ge 1}\subset(0,1)$ be summable. Define the $k$th block factor by
\begin{equation}\label{eq:eproc-blockE}
E^{(k)}(x^{m_k}):=(1-\gamma_k)\exp\bigl(m_k(d-2\varepsilon_k)\bigr)\mathbf 1\!\{\widehat Q_{m_k}(x^{m_k})\in C_k\} +\gamma_k    
\end{equation}

Clearly, this is strictly positive and depends only on the data inside the $k$th block. With this now, let us place the observations into successive disjoint blocks. Meaning, block $k$ consists of coordinates $X_{t_{k-1}+1},\dots,X_{t_k}$. Define $E_{t_0}:=1$ and for $k\ge 1$,
\[
E_{t_k}:=\prod_{j=1}^k E^{(j)}\bigl(X_{t_{j-1}+1},\dots,X_{t_j}\bigr).
\]
Then $(E_{t_k})_{k\ge 0}$ is adapted to the $(\mathcal F_{t_k})_{k\ge 0}$. Let us verify that the $e$-process property holds for every $P\in\Pcal$. Meaning, take some $P\in\Pcal$. Because $P^\infty$ is iid, the block $(X_{t_{k-1}+1},\dots, X_{t_k})$ is independent of $\mathcal F_{t_{k-1}}$ and has law $P^{m_k}$. Hence it follows that $\E_{P^\infty}[E_{t_k}\mid \mathcal F_{t_{k-1}}]= E_{t_{k-1}}\E_{P^{m_k}}\left[E^{(k)}\right]$. Therefore we can use the definition \eqref{eq:eproc-blockE} and the bound \eqref{eq:eproc-betabound} to give us that
\begin{align*}
\E_{P^{m_k}}\left[E^{(k)}\right]&=(1-\gamma_k)\exp\bigl(m_k(d-2\varepsilon_k)\bigr)P^{m_k}(A_k)+\gamma_k\\
&\le(1-\gamma_k)\exp\bigl(m_k(d-2\varepsilon_k)\bigr)\sup_{P'\in\Pcal}P'^{m_k}(A_k)+\gamma_k\\
&\le(1-\gamma_k)\cdot 1+\gamma_k=1.
\end{align*}
Then, $\E_{P^\infty}[E_{t_k}\mid\mathcal F_{t_{k-1}}]\le E_{t_{k-1}}$ almost surely. Hence $(E_{t_k})_{k\ge 0}$ is a blockwise test supermartingale as in Definition~\ref{def:test-supermartingale}. We now need to show the almost-sure growth rate under $Q^\infty$. Under $Q^\infty$, the blocks are i.i.d.\ with law $Q^{m_k}$. They are independent across $k$ but not identical since $k$ can absolutely vary. Now let $B_k$ be the event that the $k$th block belongs to $A_k$. That is,
\[
B_k:=\Bigl\{\widehat Q_{m_k}\bigl(X_{t_{k-1}+1},\dots,X_{t_k}\bigr)\in U_k\Bigr\}.
\]
By \eqref{eq:eproc-Qgood} it follows that $Q^\infty(B_k^c)\le 2^{-k}$. Henceby Borel Cantelli Lemma 1, $Q^\infty(B_k\text{ eventually})=1$. On the event that $B_k$ holds for all $k\ge K(\omega)$, for such $k$ we have that $E^{(k)}=(1-\gamma_k)\exp\bigl(m_k(d-2\varepsilon_k)\bigr)+\gamma_k$. Therefore,
\[
\log E^{(k)} \in \Bigl[m_k(d-2\varepsilon_k)+\log(1-\gamma_k),\;m_k(d-2\varepsilon_k)\Bigr],
\]
since $\log\bigl((1-\gamma_k)e^{a}+\gamma_k\bigr)\le \log(e^a)=a$, and $\log\bigl((1-\gamma_k)e^{a}+\gamma_k\bigr)\ge \log\bigl((1-\gamma_k)e^{a}\bigr)=a+\log(1-\gamma_k)$. Let us now sum over all $k\le n$ and then divide by $t_n=\sum_{j=1}^n m_j$. If we do this we get that for all large $n$ on this event that
\[
d-2\frac{\sum_{k=1}^n \varepsilon_k m_k}{t_n} +\frac{\sum_{k=1}^n \log(1-\gamma_k)}{t_n} \le \frac{1}{t_n}\log E_{t_n} \le d-2\frac{\sum_{k=1}^n \varepsilon_k m_k}{t_n} +\frac{C(\omega)}{t_n}.
\]
Note that here $C(\omega)$ is simply some finite constant which accounts for the fact that there are finitely many exceptional blocks $k<K(\omega)$. Since $t_n\to\infty$, we know that $C(\omega)/t_n\to 0$. Further, $\sum_k |\log(1-\gamma_k)|<\infty$ because $\sum_k \gamma_k<\infty$. Hence, $\sum_{k=1}^n \log(1-\gamma_k)=O(1)$ and thus $\frac{1}{t_n}\sum_{k=1}^n \log(1-\gamma_k)\to 0$. Lastly, by our choice of $(\varepsilon_k)$ in conjunction with the geometric growth of $(m_k)$ we have that
\[
\frac{\sum_{k=1}^n \varepsilon_k m_k}{t_n}\longrightarrow 0.
\]
If we take all that we have shown now together we have $\frac{1}{t_n}\log E_{t_n}\to d$ almost surely under $Q^\infty$. It remains to now show the expected log-growth rate under $Q^\infty$. Because we always have $\log E_{t_n}=\sum_{k=1}^n \log E^{(k)}$, it follows that by linearity, $\E_{Q^\infty}[\log E_{t_n}]=\sum_{k=1}^n \E_{Q^{m_k}}[\log E^{(k)}]$. On $B_k$ we have that $\log E^{(k)}\ge m_k(d-2\varepsilon_k)+\log(1-\gamma_k)$. And on $B_k^c$, $\log E^{(k)}=\log\gamma_k$. Therefore we get that
\[
\E_{Q^{m_k}}[\log E^{(k)}] \ge Q^{m_k}(A_k)\Bigl(m_k(d-2\varepsilon_k)+\log(1-\gamma_k)\Bigr) +Q^{m_k}(A_k^c)\log\gamma_k.
\]
Recall that from \eqref{eq:eproc-Qgood}, $Q^{m_k}(A_k)\ge 1-2^{-k}$. Let us also use the two facts that $\log(1-\gamma_k)=O(\gamma_k)$ with $\sum_k\gamma_k<\infty$, and $\sum_k 2^{-k}|\log\gamma_k|<\infty$. Hence,
\[
\frac{1}{t_n}\E_{Q^\infty}[\log E_{t_n}] \ge d-2\frac{\sum_{k=1}^n \varepsilon_k m_k}{t_n} - o(1).
\]
Separately, $(E_{t_k})$ is a blockwise $e$-process for $\Pcal$. Hence it follows that each $E_{t_k}$ is a block $e$-variable for the null $\Pcal$ at block time $k$. Thus for every $P\in\Pcal$ we have that $\E_{P^\infty}[E_{t_k}]\le 1$. Hence by Lemma~\ref{lem:entropy-ineq} applying $\mu=Q^{t_k}$, $\nu=P^{t_k}$ and $f=\log E_{t_k}$ (so $e^f=E_{t_k}$) we get that
\[
\E_{Q^\infty}[\log E_{t_k}]\le \KL(Q^{t_k}\|P^{t_k})+\log \E_{P^\infty}[E_{t_k}].
\]
Necessarily $(E_{t_k})$ is an $e$-process for $\Pcal$, which means that $\E_{P^\infty}[E_{t_k}]\le 1$. Therefore, $\E_{Q^\infty}[\log E_{t_k}]\le \KL(Q^{t_k}\|P^{t_k})$. Now if we take the $\inf_{P\in\Pcal}$ and use tensorization, $\E_{Q^\infty}[\log E_{t_k}]\le t_k d$.
Dividing by $t_k$ and combining with the lower bound gives us equality; that is, $\lim_{k\to\infty}\frac{1}{t_k}\E_{Q^\infty}[\log E_{t_k}]=d$. This completes the proof.
\end{proof}

\subsection{Near log-optimal blockwise test supermartingale for a weakly compact $\Qcal$}
For a single alternative $Q$, Proposition~\ref{prop:eprocess-rate} gives us a blockwise test supermartingale with asymptotic log-growth exactly equal to the $\KLinf(Q,\Pcal)$, under weak lower semicontinuity at $Q$. A question that follows from this is whether we can build a single blockwise test supermartingale in the composite alternative case. Now, recall that $d(Q):=\KLinf(Q,\Pcal)$. The question becomes can we build a single blockwise test-supermartingale\footnote{Note, that this depends on $(\Pcal, \Qcal,\varepsilon)$ only.} that achieves a pointwise near optimal growth rate $d(Q)-\varepsilon$ for every $Q\in\Qcal$.

One potential idea to do this is if we simply mix over some $\varepsilon$ net of $\Qcal$. However, the subtlety here is that $d(Q)$ can vary significantly across $\Qcal$. Furthermore, note that in general it is only lower semicontinuous rather than continuous. So, we must partition the alternative class by rate levels and cover each rate level separately. To this end, we first construct a neighborhood-based typical-set $e$-process at some prescribed rate.

\begin{lemma}\label{lem:typicalset-eprocess-fixedU}
Let $\Pcal\subseteq\mathcal M_1(\X)$ be nonempty. Then let $U\subseteq\mathcal M_1(\X)$ be weakly open. Assume in addition that $\cl(U)$ is convex. Then, set
\[
b(U):=\inf_{R\in\cl(U)}\Phi(R)=\inf_{R\in\cl(U)}\inf_{P\in\Pcal}\KL(R\|P)\in[0,\infty].
\]

\noindent Now assume that $b(U)>0$ and take some $r\in(0, b(U))$. Then it follows that there exist deterministic block lengths $(m_k)_{k\ge 1}$ with $t_k:=\sum_{j=1}^k m_j\to\infty$ and a test supermartingale $(E_{t_k})_{k\ge 0}$ for $\Pcal$ such that for every $Q\in U$ we have $Q^\infty$-almost-surely that
\[
\lim_{k\to\infty}\frac{1}{t_k}\log E_{t_k} = r.
\]
In addition,
\[
\lim_{k\to\infty}\frac{1}{t_k}\E_{Q^\infty}[\log E_{t_k}] = r.
\]
\end{lemma}

For this lemma, we will need Sanov's asymptotic theorem and to understand how we define Q-typical sets. We will build that up as follows. Suppose that $X_1, X_2,\dots$ are coordinate maps on $\X^\mathbb{N}$. Now, for $n\in\mathbb{N}$, we know that the empirical measure of a sample $x^n=(x_1,\dots,x_n)\in X^n$ is simply,
\[
\widehat{Q}_n(x^n) :=\frac{1}{n}\sum_{i=1}^n \delta_{x_i} \in\mathcal{M}_1(\X).
\]

\begin{proof}[Proof of Lemma~\ref{lem:typicalset-eprocess-fixedU}]
Consider some summable sequence $(\gamma_k)_{k\ge 1}\subset(0,1)$. An apt choice could be for example $\gamma_k=2^{-k}$. In addition, take some particular deterministic block lengths with $m_k\to\infty$, for example, $m_k=k^2$. Now, let $C:=\cl(U)$. This is, convex and weakly closed by assumption. In addition, let $b(U)=\inf_{R\in C}\Phi(R)$. Now for each $k$, define $A_k:=\{\widehat Q_{m_k}\in U\}$ and $\beta_k:=\sup_{P\in\Pcal}P^{m_k}(A_k)$. Because of the fact that $A_k\subseteq\{\widehat Q_{m_k}\in C\}$, Lemma~\ref{lem:convex-empirical} gives us that for every $P\in\Pcal$ that
\[
P^{m_k}(A_k)\le P^{m_k}(\widehat Q_{m_k}\in C) \le \exp\Bigl(-m_k\inf_{R\in C}\KL(R\|P)\Bigr).
\]
Therefore it follows that $\beta_k\le \exp\Bigl(-m_k\inf_{R\in C}\Phi(R)\Bigr)=\exp\bigl(-m_k b(U)\bigr)$.
Clearly, since $r<b(U)$ by assumption, there exists $\eta>0$ such that $r+2\eta<b(U)$. Therefore there must exist $k_0$ such that for all $k\ge k_0$, $\beta_k\le \exp(-m_k(r+\eta))$. Now let us define the $k$th block $e$-variable as
\[
E^{(k)}:=
\begin{cases}
1, & k<k_0,\\[0.4em]
(1-\gamma_k)\exp(m_k r)\one_{A_k}+\gamma_k, & k\ge k_0.
\end{cases}
\]
Now, for any $P\in\Pcal$ and $k\ge k_0$ we have that
\begin{align*}
\E_{P^{m_k}}[E^{(k)}]
&=(1-\gamma_k)\exp(m_k r)P^{m_k}(A_k)+\gamma_k\\
&\le (1-\gamma_k)\exp(m_k r)\beta_k+\gamma_k\\
&\le (1-\gamma_k)\exp(-m_k\eta)+\gamma_k\\
&\le (1-\gamma_k)\cdot 1+\gamma_k=1.
\end{align*}
We know that $\E_{P^{m_k}}[E^{(k)}]=1$ for $k<k_0$. Therefore the blockwise product $E_{t_0}:=1$ and
\[
E_{t_k}:=\prod_{j=1}^k E^{(j)}\bigl(X_{t_{j-1}+1},\dots,X_{t_j}\bigr),
\]
is in fact a test supermartingale $\Pcal$, exactly as in Proposition~\ref{prop:eprocess-rate}. Now take some $Q\in U$. Because $U$ is weakly open and contains $Q$, the complement $U^c$ is weakly closed and does not contain $Q$. In other words, if there existed a sequence $(R_n)\subset U^c$ with $\KL(R_n\|Q)\downarrow 0$, by Pinsker's inequality, $R_n\Rightarrow Q$ is forced to occur. But that contradicts the fact that $U^c$ is weakly closed and excludes $Q$. So it must be the case that $c_Q(U):=\inf_{R\in\cl(U^c)}\KL(R\|Q)$ satisfies $c_Q(U)>0$. Now let us use Sanov's theorem under $Q$. Doing this gives us
\[
\limsup_{m\to\infty}\frac{1}{m}\log Q^{m}(\widehat Q_m\in U^c) \le -c_Q(U).
\]
In particular, $Q^{m_k}(A_k^c)\le \exp(-m_k c_Q(U)/2)$ for all large $k$. Hence it follows that $\sum_k Q^{m_k}(A_k^c)<\infty$. Then thanks to Borel Cantelli Lemma 1, $A_k$ occurs eventually $Q^\infty$ almost surely. Let us work on the event that $A_k$ holds for all sufficiently large $k$. On this event we have that for $k\ge k_0$,
\[
\log E^{(k)}\in\Bigl[m_kr+\log(1-\gamma_k),\; m_kr\Bigr].
\]
Therefore, summing and dividing by $t_k=\sum_{j\le k}m_j$, we have that
\[
r+\frac{1}{t_k}\sum_{j=1}^k\log(1-\gamma_j)\le \frac{1}{t_k}\log E_{t_k} \le r+\frac{C(\omega)}{t_k}.
\]
Here, recall that $C(\omega)$ is just some finite random constant for finitely many exceptional blocks. However, since $\sum_j|\log(1-\gamma_j)|<\infty$ and $t_k\to\infty$, both of the correction terms must vanish. Therefore we get that almost surely under $Q^\infty$,
\[
\frac{1}{t_k}\log E_{t_k} \to r,
\]
which was our first claim. In addition, the expected log growth claim follows exactly as in Proposition~\ref{prop:eprocess-rate}. Meaning, from linearity of expectation, $\E[\log E_{t_k}]=\sum_{j\le k}\E[\log E^{(j)}]$. Then, the Borel-Cantelli lemma gives us a lower bound $\E[\log E^{(k)}]\ge (1-o(1))m_kr+O(1)$. Moreover, $\log E^{(k)}\le m_k r$ for all $k$, hence $\log E_{t_k}\le rt_k$. As such, $\E_{Q^\infty}[\log E_{t_k}]\le rt_k$. And this in conjunction with the lower bound tells us that indeed $\E_{Q^\infty}[\log E_{t_k}]/t_k\to r$.
\end{proof}

\noindent Now, we are ready for our main theorem for weakly compact $\Qcal$. Our theorem will prove the existence of a pointwise near-optimal composite blockwise test supermartingale through rate level covers.

For intuition, if we take $\Qcal=\{Q\}$, thanks to Theorem~\ref{thm:cover-and-mix-eprocess} below, we get for every $\varepsilon>0$ a test supermartingale with
\[
\liminf_{k\to\infty}\frac{1}{t_k}\log E_{t_k} \ge d(Q)-\varepsilon,
\]
almost surely under $Q^{\infty}$. And the analogous expected log bound will also hold. So it will recover the lower bound direction of Proposition~\ref{prop:eprocess-rate} in the singleton case. However, Proposition~\ref{prop:eprocess-rate} will be much sharper as it gives an exact limit $d(Q)$ rather than a near-optimal lower bound.

\begin{theorem}\label{thm:cover-and-mix-eprocess}
Assume that $\X$ is Polish and that $\Pcal\subseteq\mathcal M_1(\X)$ is weakly compact. Also, let $\Qcal\subseteq\mathcal M_1(\X)$ be weakly compact. Assume that $\Qcal\subseteq\Pcal^c$. As before, let us define $d(Q):=\KLinf(Q,\Pcal)=\Phi(Q)\in(0,\infty]$ for $Q\in\Qcal$. Now take some $\varepsilon>0$. Then it follows that there exist deterministic block lengths $(m_k)$ with $t_k:=\sum_{j\le k}m_j\to\infty$ and a \underline{test-supermartingale} $(E_{t_k})_{k\ge 0}$ such that for every $Q\in\Qcal$ with $d(Q)<\infty$,
\[
\boxed{
\liminf_{k\to\infty}\frac{1}{t_k}\log E_{t_k} \ge d(Q)-\varepsilon \quad Q^\infty\text{-a.s.}
}
\]
In addition,
\[
\boxed{
\liminf_{k\to\infty}\frac{1}{t_k}\E_{Q^\infty}[\log E_{t_k}] \ge d(Q)-\varepsilon.
}
\]
\end{theorem}

\begin{proof}
Take a particular $\varepsilon>0$. For $m\ge 1$, let us define the rate grid as
\[
r_m:=\frac{m\varepsilon}{2},\qquad s_m:=r_m-\frac{\varepsilon}{2}=\frac{(m-1)\varepsilon}{2}.
\]
Then, for each $m\ge 1$ let us define the superlevel set as $K_m:=\{Q\in\Qcal:d(Q)\ge r_m\}$. Now, since $\Pcal$ is weakly compact by assumption, by \parencite[Lemma~3]{ram2026powersequentialtestsexist} it follows that $\Phi$ is weakly lower semicontinuous everywhere. Therefore, $d=\Phi$ is weakly lower semicontinuous on $\mathcal M_1(\X)$. So $K_m$ must be weakly closed in $\Qcal$, and since $\Qcal$ is weakly compact that must mean that each of the $K_m$ are weakly compact. Now take some particular but arbitrarily chosen $m\ge 1$. For each $Q\in K_m$ we have that $d(Q)\ge r_m$. Hence that in conjunction with weak lower semicontinuity of $\Phi$ at $Q$ tells us that the set $O_{m,Q}:=\{R\in\mathcal M_1(\X):\Phi(R)>r_m-\varepsilon/4\}$ is weakly open and contains $Q$. Recall that $\X$ was Polish, and so it follows that the weak topology is metrizable on $\mathcal M_1(\X)$. Now consider the bounded Lipschitz metric $d_{\mathrm{BL}}$. Recall that $O_{m,Q}$ is weakly open and contains $Q$. Therefore, we can choose $\delta_{m,Q}>0$ such that
\[
\overline B^{\mathrm{BL}}_{\delta_{m,Q}}(Q)\subseteq O_{m,Q}.
\]

Now, let us set $U_{m,Q}:=B^{\mathrm{BL}}_{\delta_{m,Q}}(Q)$. It follows that $\cl(U_{m,Q})=\overline B^{\mathrm{BL}}_{\delta_{m,Q}}(Q)$ is convex and weakly closed. In addition,
\[
\inf_{R\in\cl(U_{m,Q})}\Phi(R)\ge r_m-\varepsilon/4.
\]
The family $\{U_{m, Q}:Q\in K_m\}$ is an open cover of the compact set $K_m$, so it must have a finite subcover. Let us denote this subcover as
\[
K_m\subseteq \bigcup_{j=1}^{J_m}U_{m,j},
\]
where for $j=1,\dots,J_m$, we have that $\inf_{R\in\cl(U_{m,j})}\Phi(R)\ge r_m-\varepsilon/4$. Because $s_m=r_m-\varepsilon/2<r_m-\varepsilon/4\le \inf_{R\in\cl(U_{m,j})}\Phi(R)$, we can use Lemma~\ref{lem:typicalset-eprocess-fixedU} with $U=U_{m,j}$ and $r=s_m$. To this end, we get for each pair $(m, j)$ a test-supermartingale $(E^{(m,j)}_k)_{k\ge 0}$ for $\Pcal$ (and note that, they are all constructed on the same block schedule; for example by taking the same $m_k=k^2$ in the lemma). And such that this blockwise test-supermartingale satisfies for every $Q\in U_{m,j}$,
\[
\lim_{k\to\infty}\frac{1}{t_k}\log E^{(m,j)}_{t_k} = s_m,
\]
almost surely under $Q^\infty$, and likewise for the expected log-growth necessarily. Now let us choose weights $(w_{m, j})$ such that $\sum_{m\ge 1}\sum_{j=1}^{J_m}w_{m,j}\le 1$. Just to be concrete let us set $w_{m,j}:=\frac{2^{-m}}{J_m}$. Now then let us define the mixed process where $k\ge 0$ as
\[
E_{t_k}:=\sum_{m\ge 1}\sum_{j=1}^{J_m} w_{m,j}E^{(m,j)}_{t_k}.
\]
Each of the $E^{(m,j)}$ is a blockwise $e$-process for $\Pcal$. And, conditional expectation is linear, which means that $(E_{t_k})$ is also a blockwise $e$-process for $\Pcal$. Meaning, $\E_{P^\infty}[E_{t_k}\mid\mathcal F_{t_{k-1}}]\le E_{k-1}$ for all $P\in\Pcal$. Now consider some particular $Q\in\Qcal$ with $d(Q)<\infty$. Let us choose $m$ so that $r_m\le d(Q)<r_{m+1}=r_m+\varepsilon/2$. Clearly then by that choice, it follows that $Q\in K_m$, which means $Q\in U_{m,j}$ for some $j\in\{1,\dots,J_m\}$. For this particular $(m,j)$ we have that pointwise $E_{t_k}\ge w_{m,j}E^{(m,j)}_{t_k}$, and so it follows that $\log E_{t_k}\ge \log w_{m,j}+\log E^{(m,j)}_{t_k}$. All we need to do now is divide by $t_k$ and take the $\liminf$ and we get that almost surely under $Q^\infty$,
\[
\liminf_{k\to\infty}\frac{1}{t_k}\log E_{t_k} \ge \lim_{k\to\infty}\frac{1}{t_k}\log E^{(m,j)}_{t_k} = s_m,
\]
as $\log(w_{m,j})/t_k\to0$. In addition, note that our choice of $m$ above implies that $s_m=r_m-\varepsilon/2 \ge d(Q)-\varepsilon$. This proves the almost sure lower bound. The expected log bound follows with the same inequality applied and taking expectations. That is,
\[
\E_{Q^\infty}[\log E_{t_k}] \ge \log w_{m,j}+\E_{Q^\infty}[\log E^{(m,j)}_{t_k}].
\]
Then, dividing by $t_k$ and taking the $\liminf$ and using the fact that $\E[\log E^{(m,j)}_k]/t_k\to s_m$ we get that
\[
\liminf_{k\to\infty}\frac{1}{t_k}\E_{Q^\infty}[\log E_{t_k}]\ge s_m\ge d(Q)-\varepsilon.
\]
This completes the proof.
\end{proof}

We will now give a uniform version of this theorem. 

\begin{corollary}\label{cor:uniform-cover-and-mix}
Under the same setup as Theorem~\ref{thm:cover-and-mix-eprocess}, define,
\[
d_{\min}:=\inf_{Q\in\Qcal} d(Q)=\inf_{Q\in\Qcal}\KLinf(Q,\Pcal)\in(0,\infty].
\]
Then it follows that for every $\varepsilon>0$, wp 1 under $Q^\infty$ for every $Q\in\Qcal$ with $d(Q)<\infty$, the same test supermartingale will satisfy
\[
\boxed{
\liminf_{k\to\infty}\frac{1}{t_k}\log E_{t_k} \ge d_{\min}-\varepsilon.
}
\]
And it also satisfies
\[
\boxed{
\liminf_{k\to\infty}\frac{1}{t_k}\inf_{Q\in\Qcal}\E_{Q^\infty}[\log E_{t_k}]\ge d_{\min}-\varepsilon.
}
\]
\end{corollary}

\begin{proof}
For each $Q\in\Qcal$, thanks to Theorem~\ref{thm:cover-and-mix-eprocess} we have that with probability one under $Q^\infty$,
\[
\liminf_{k\to\infty}\frac{1}{t_k}\log E_{t_k} \ge d(Q)-\varepsilon.
\]
Now, for every $Q\in\Qcal$, $d(Q)\ge d_{\mathrm{min}}$. So our almost-sure claim is immediate. The expected-log claim follows the same way: just from taking the infimum over $Q\in\Qcal$ in the lower bound from Theorem~\ref{thm:cover-and-mix-eprocess}.
\end{proof}

Throughout this paper, we have used blockwise constructions for achievability (through blockwise test supermartingales on a reduced filtration). For impossibility, it holds for arbitrary $e$-processes. For sequential testing, the processes on $\bbF_{\bt}$ are enough already, since if we threshold at times $t_k$, we have a stopping rule in the original filtration. Thus, albeit the evidence process still updates on a schedule, the power-one implications we prove are still meaningful. We will formalize this more as follows.

\begin{remark}\label{remark:formalize-blockwise}
Our blockwise constructions we use throughout this paper are structurally important. They allow us to convert a finite horizon $e$-variable into a sequentially valid process. Note that by duality at horizon $m$, we have an $m$-sample $e$-variable $E^{\star}:\X^m\to[0,\infty]$, $E^{\star}\in(\Pcal^m)^{\circ}$. It has a large expected log under the target law $Q^m$. However, note that this is not an $e$-process on $\bbF$, rather it is terminal at time $m$. However, our blockwise construction solves this problem. We partition the data into disjoint lengths of length $m$. We define $t_k:=km$ and set
\[
E_{t_k}:=\prod_{j=1}^k E^\star\bigl(X_{(j-1)m+1},\dots,X_{jm}\bigr).
\]
Then we have for every $P\in\Pcal$ that
\[
\E_{P^\infty}[E_{t_k}\mid \mathcal F_{t_{k-1}}] = E_{t_{k-1}}\E_{P^m}[E^\star] \le E_{t_{k-1}}.
\]
This holds since the future block is independent of the past and has law $P^m$, and since $E^{\star}\in (\Pcal^m)^{\circ}$. So, $(E_{t_k})_{k\ge 0}$ is a test supermartingale on the reduced filtration $\mathcal F_{t_k}$. Under $Q^{\infty}$, the block log-increments are i.i.d.\ and hence both the expected log growth and almost sure growth are immediate.

In addition, one cannot in general hold the process constant between $t_k$ and $t_{k+1}$. As indicated in Example~\ref{ex:counterzero-versus-constant-eproc}, a $\bbF$-stopping rule can exploit the information that is revealed inside the current block, which will violate the $e$-process property. This can be saved by setting the process equal to $0$ off the block times, and it will still be a valid process. However, this will only preserve in full-generality the $\limsup$ (of the expected $\log$ objective) but will not preserve any almost sure or full $\liminf$ statement.

One should therefore look at blockwise processes as the constructive objects for lower-bound and achievability arguments. These will certainly be $e$-processes, but on a reduced filtration. However, this is already sufficient for sequential testing applications since $\tau_{\alpha}:=\inf\{t_k:E_{t_k}\ge 1/\alpha\}$ is a $(\mathcal F_n)$-stopping rule in the original filtration.
\end{remark}

\section{Summary}\label{sec:cute}
In this work, we proved that the intrinsic maximal growth rate for any sequential procedure under absolutely no assumptions is the normalized limit of the $\infKL$, rather than the $\KLinf$, the latter in general being larger. We provided a necessary and sufficient condition for the existence of power one sequential tests (in terms of positivity of the $k$-step bipolar rate) for composite $\Pcal$ and singleton $Q$, greatly generalizing prior work. We provided general conditions under which  the maximal achievable growth rate $\limsup_{n\to\infty}\frac1n\infKL(Q^n,\Pcal^n)$ actually equals $\KLinf$. We also proved the maximal intrinsic growth rate in the composite alternative setting (under no assumptions whatsoever), which turns out to be a limiting variant of the $\mathrm{GROW}$ value, and provide some general conditions under which this reduces to a $\KL$ type expression. In addition, we also provided general conditions for the existence of power one sequential tests for composite alternatives and related that to strict positivity of the intrinsic maximal growth rate in the composite case. 


\subsection*{Acknowledgments}
Aaditya Ramdas was supported by NSF DMS grant 2310718.

\printbibliography

\appendix

\section{Omitted Proofs}\label{sec:proofs-apendix}
\begin{proof}[Proof of Lemma~\ref{lem:eprocess-upper-by-infkl}]
Consider a particular $k\ge 1$ and abbreviate $n:=t_k$. By definition, $W_{t_k}$ is $\mathcal F_{t_k}=\sigma(X_1,\dots,X_n)$-measurable. By the Doob-Dynkin lemma, we thus have a measurable map $w_k:\X^{n}\to[0,\infty]$ such that $W_{t_k}=w_k(X_1,\dots,X_n)$. Because $\tau\equiv t_k$ is a $(\bbF_{\bt})$-stopping time, Definition~\ref{def:eprocess} implies for every $P\in\Pcal$ that $E_{P^\infty}[W_{t_k}]\le 1$. Under $P^\infty=P^{\otimes\mathbb N}$, the vector $(X_1,\dots,X_n)$ has law $P^n$, so for all $P\in\Pcal$, $\E_{P^\infty}[W_{t_k}]=\E_{P^n}[w_k]\le 1$. Therefore, $w_k\in(\Pcal^n)^\circ$. By definition~\ref{def:bipolar} for every $R\in(\Pcal^n)^{\circ\circ}$ we therefore have that
\begin{equation}\label{eq:ERwk}
\E_R[w_k]\le 1.    
\end{equation}

Now take some arbitrary $R\in(\Pcal^n)^{\circ\circ}$. If $\KL(Q^n\|R)=+\infty$, then $\E_{Q^n}[\log w_k]\le \KL(Q^n\|R)$. So, we may assume for the rest of this proof without loss of generality that $\KL(Q^n\|R)<\infty$. As such, $Q^n\ll R$, so the density $f:=\frac{dQ^n}{dR}$ exists. By Definition~\ref{def:kl} we have that
\begin{equation}\label{eq:KL-as-logf}
\KL(Q^n\|R)=\int_{\X^n}\log(f)dQ^n.   
\end{equation}

In addition, $f>0$ holds $Q^n$-almost everywhere. Consequently, $\log w_k=\log(\frac{w_k}{f})+\log f$ also holds $Q^n$ almost everywhere. (And certainly $\log 0=-\infty$ is allowed.) Integrating with respect to $Q^n$ and using \eqref{eq:KL-as-logf} gives
\begin{align}
\E_{Q^n}[\log w_k]&=\int_{\X^n}\log\Bigl(\frac{w_k}{f}\Bigr)dQ^n + \KL(Q^n\|R).
\label{eq:split}
\end{align}

Note that $\log x\le x-1$ for all $x\ge 0$ (noting that $\log 0=-\infty$), where here we use $x=\frac{w_k}{f}$. In addition, we also recall that $dQ^n=fdR$. Therefore, we obtain
\begin{align}
\int_{\X^n}\log\Bigl(\frac{w_k}{f}\Bigr)dQ^n &\le \int_{\X^n}\Bigl(\frac{w_k}{f}-1\Bigr)dQ^n \nonumber\\
&= \int_{\X^n} w_k\,dR-1 \nonumber\\ 
&\le 0  
\label{eq:log-bound}.
\end{align}

Note that the last inequality here used \eqref{eq:ERwk}. Thus, taking together \eqref{eq:split} and \eqref{eq:log-bound} we have $\E_{Q^n}[\log w_k]\le \KL(Q^n\|R)$. This holds for every $R\in(\Pcal^n)^{\circ\circ}$. Therefore,
\[
\E_{Q^n}[\log w_k] \le \inf_{R\in(\Pcal^n)^{\circ\circ}}\KL(Q^n\|R) = a_n = a_{t_k}.
\]
By definition, $W_{t_k}=w_k(X_1,\dots,X_n)$ depends only on the first $n$ samples. So, $\E_{Q^\infty}[\log W_{t_k}]=\E_{Q^n}[\log w_k]$, completing the proof.
\end{proof}

\begin{proof}[Proof of Theorem~\ref{thm:powerone-characterization}]
For each $n\ge 1$, let us recall that
\[
a_n(Q):=\inf_{R\in(\Pcal^n)^{\circ\circ}}\KL(Q^n\|R)\in[0,\infty].
\]
Now, the implication of \textcolor{red}{(1)} implying \textcolor{red}{(2)} comes immediately thanks to Proposition~\ref{prop:powerone-implies-out}. Similarly, \textcolor{red}{(2)} implies \textcolor{red}{(3)} comes from Proposition~\ref{prop:notinbipolar-positive}. And, \textcolor{red}{(3)} implying \textcolor{red}{(2)} is also immediate. That is, if $Q^{k}\in (\Pcal^k)^{\circ\circ}$, it follows that $Q^{k}$ is feasible in the infimum that defines $a_k(Q)$, and so $a_k(Q)\le \KL(Q^k\|Q^k)=0$, proving the claim. So it remains for us to show that \textcolor{red}{(2)} implies \textcolor{red}{(1)}. Assume that $Q^k\notin (\Pcal^k)^{\circ\circ}$ for some $k$. By definition of the bipolar, there exists some $E\in(\Pcal^k)^{\circ}$ such thar $\E_{Q^k}[E]>1$. Let us replace $E$ by $E\wedge M$ for large enough $M$, and through this we can assume that $E$ is bounded and still will satisfy $\E_{Q^k}[E]>1$. Now, for some $\lambda\in(0,1)$ let us define $Z_{\lambda}:=1-\lambda+\lambda E$. Then, it follows that $Z_{\lambda}\in (\Pcal^k)^{\circ}$, since for every $P\in\Pcal$,
\[
\E_{P^k}[Z_\lambda] = 1-\lambda+\lambda\E_{P^k}[E]\le 1.
\]
In addition, $0<1-\lambda\le Z_{\lambda}\le 1-\lambda+\lambda M$, hence $\log Z_{\lambda}$ is bounded. Now define $\psi(\lambda):=\E_{Q^k}[\log Z_{\lambda}]$. Now, $E$ is bounded, $\psi$ is differentiable at $0$, and $\psi'(0)=\E_{Q^k}[E]-1>0$. Therefore, for sufficiently small enough $\lambda\in (0,1)$, we have that $m:=\E_{Q^k}[\log Z_{\lambda}]>0$. Let us now partition the samples into successive disjoint blocks of length $k$. We will let $S_0:=1$. And for $j\ge 1$,
\[
S_j:=\prod_{i=1}^jZ_\lambda(X_{(i-1)k+1},\dots,X_{ik}).
\]
Now, for every $P\in\Pcal$, the process $(S_j)_{j\ge 0}$ is a nonnegative supermartingale in block time. Because, each block factor will have $P^{k}$ expectation at most $1$. Therefore by Ville's inequality for every $\alpha\in(0,1)$ it follows that
\[
P^\infty\left(\sup_{j\ge 0} S_j\ge \frac{1}{\alpha}\right)\le \alpha.
\]
Now, define the stopping rule $\tau_{\alpha}:=\inf\{jk: S_j\ge 1/\alpha\}$. Therefore, $\tau_{\alpha}$ is a $(\mathcal F_n)$-stopping time and $\sup_{P\in\Pcal}P^{\infty}(\tau_\alpha<\infty)\le\alpha$. Now, under $Q^\infty$, the random variables
\[
Y_j:=\log Z_\lambda(X_{(j-1)k+1},\dots,X_{jk})
\]
are iid and bounded with mean $m>0$. So by the strong law we have that $Q^{\infty}$ almost surely,
\[
\frac{1}{j}\log S_j = \frac{1}{j}\sum_{i=1}^j Y_i \longrightarrow m>0.
\]
Therefore, since with probability 1 under $Q^{\infty}$, $S_j\to\infty$, it follows that $Q^{\infty}(\tau_{\alpha}<\infty)=1$. This proves the implication \textcolor{red}{(2)} implies \textcolor{red}{(1)}. To show the implication with \textcolor{red}{(4)}, assume that for some $k$, $a_k(Q)>0$. Now, from superadditivity, for $m\ge 1$, it is clear that $a_{mk}(Q)\ge ma_k(Q)$. Therefore, $d_{\star}(Q,\Pcal)\ge a_k(Q)/k >0$, proving that \textcolor{red}{(3)} implies \textcolor{red}{(4)}. Now to show necessity, assume that \textcolor{red}{(4)} holds, meaning suppose that $d_{\star}(Q,\Pcal)>0$. This means that $a_k(Q)>0$ for some $k$, so \textcolor{red}{(4)} implying \textcolor{red}{(3)} is obvious. Lastly, our final remark in the lemma follows thanks to Corollary~\ref{cor:out-propogates} combined with applying Proposition~\ref{prop:notinbipolar-positive}.
\end{proof}

\begin{proof}[Proof of Lemma~\ref{lem:an-superadditive-sec}]
Let us start by taking some particular $Q\in\mathcal M_1(\X)$ and $n,m\in\mathbb N$. We will first deal with the case where both $a_n(Q)<\infty$ and $a_m(Q)<\infty$. Let us let $\varepsilon>0$. Thanks to Proposition~\ref{prop:infkl-e-power} it follows that there exist $E_n\in(\Pcal^n)^{\circ}$ and $F_m\in(\Pcal^m)^{\circ}$ such that $\E_{Q^n}[\log E_n]\ge a_n(Q)-\varepsilon$ and $\E_{Q^m}[\log F_m]\ge a_m(Q)-\varepsilon$. Now, having explicated this, let us define $G: \X^{n+m}\to[0,\infty]$ by
\[
G(x_1,\dots,x_{n+m}):=E_n(x_1,\dots,x_n)F_m(x_{n+1},\dots,x_{n+m}).
\]
Our claim here will be that $G$ is in the polar of $\Pcal^{n+m}$. To see that note that for each $P\in\Pcal$, the product structure of $P^{n+m}=P^{n}\otimes P^{m}$ tells us that
\[
\E_{P^{n+m}}[G]=\E_{P^{n+m}}\left[E_n(x_1,\dots,x_n)F_m(x_{n+1},\dots,x_{n+m})\right]=\E_{P^n}[E_n]\E_{P^m}[F_m]\le 1\cdot 1\le 1.
\]
This follows from the independence of the first $n$ coordinates and the last $m$ coordinates under $P^{n+m}$. And, the inequality follows because $E_n\in(\Pcal^n)^{\circ}$ and $F_m\in(\Pcal^m)^{\circ}$. So we get that $G\in(\Pcal^{n+m})^{\circ}$, which was our claim. Hence once again thanks to Proposition~\ref{prop:infkl-e-power} we get
\begin{align}
a_{n+m}(Q)&=\sup_{E\in(\Pcal^{n+m})^{\circ}}\E_{Q^{n+m}}[\log E]\notag\\
&\ge \E_{Q^{n+m}}[\log G]\notag\\
&=\E_{Q^{n+m}}\Big[\log E_n(X_1,\dots,X_n)+\log F_m(X_{n+1},\dots,X_{n+m})\Big]\notag\\
&=\E_{Q^n}[\log E_n]+\E_{Q^m}[\log F_m]\notag\\
&\ge a_n(Q)+a_m(Q)-2\varepsilon.\label{eq:superadd}
\end{align}
Here, in \eqref{eq:superadd}, we once again use independence under $Q^{n+m}=Q^n\otimes Q^m$ along with the fact that the former part only will depend on the first $n$ coordinates and the latter on the last $m$. And so since $\varepsilon>0$ was arbitrary, we may indeed conclude superadditivity whenever both $a_n(Q)$ and $a_m(Q)$ are finite. That is,
\[
a_{n+m}(Q)\ge a_n(Q)+a_m(Q).
\]
Suppose one of $a_m(Q)$ or $a_n(Q)$ were $+\infty$. In this case, our same argument here applies, but using lower targets that are arbitrarily large instead of $a_m(Q)-\varepsilon$ or $a_n(Q)-\varepsilon$. Therefore, it follows that $(a_n(Q))_{n\ge 1}$ is superadditive in any case. So the existence of the limit is immediate now from Fekete's lemma. Meaning,
\[
\lim_{n\to\infty} \frac{a_n(Q)}{n}=\sup_{n\ge 1}\frac{a_n(Q)}{n}.
\]
This completes the proof.
\end{proof}

\begin{proof}[Proof of Lemma~\ref{lem:bn-le-an}]
Take some $n\in\mathbb N$ and take any $E\in(\Pcal^n)^\circ$. For every $Q\in\Qcal$,
\[
\inf_{Q'\in\Qcal}\E_{Q'^n}[\log E]\le \E_{Q^n}[\log E] \le \sup_{E'\in(\Pcal^n)^\circ}\E_{Q^n}[\log E'].
\]
Then, by Proposition~\ref{prop:infkl-e-power}, i.e.\ the finite horizon strong duality, we know that the rightmost term in our above expression is simply $a_n(Q)=\infKL(Q^n,\Pcal^n)$. If we then take the $\inf_{Q\in\Qcal}$ over the inequality we therefore get that
\[
\inf_{Q'\in\Qcal}\E_{Q'^n}[\log E]\le \inf_{Q\in\Qcal} a_n(Q) = a_n(\Qcal).
\]
Finally, if we take the $\sup_{E\in(\Pcal^n)^\circ}$ over the lhs we get that $b_n(\Qcal,\Pcal)\le a_n(\Qcal)$. Now all we need to do is divide by $n$ and take the $\limsup_{n\to\infty}$. This will certainly preserve the inequality. And as such, we will get that $\overline d_{\mathrm{rob}}\le \overline d_{\mathrm{wc}}$. 
\end{proof}

\begin{proof}[Proof of Corollary~\ref{cor:bipolar-one-step-positive}]
First, note that $\KLinf(Q,\Pcal^{\circ\circ})=\inf_{R\in\Pcal^{\circ\circ}}\KL(Q\|R)=a_1(Q)$. Let us set $c:=\inf_{Q\in\Qcal} a_1(Q)=\inf_{Q\in\Qcal}\KLinf(Q,\Pcal^{\circ\circ})$. By assumption, it follows that $c>0$. Now let us take some particular $Q\in\Qcal$. Thanks to Lemma~\ref{lem:an-superadditive-sec}, it follows that the sequence $(a_n(Q))_{n\ge 1}$ is superadditive. So we get that for every $n\ge 1$,
\begin{align}
a_n(Q)&\ge a_{n-1}(Q)+a_1(Q)\notag\\
&\ge a_{n-2}(Q)+2a_1(Q)\notag\\
&\,\,\vdots\notag\\
&\ge na_1(Q)\notag\\
&\ge nc.
\label{eq:an-lower-by-al}
\end{align}
Then taking the infimum over all $Q\in\Qcal$ gives us that $a_n(\Qcal)=\inf_{Q\in\Qcal}a_n(Q)\ge nc$. Therefore,
\[
\Gamma(\Qcal,\Pcal)=\overline d_{\mathrm{wc}}=\limsup_{n\to\infty}\frac1n a_n(\Qcal)\ge \limsup_{n\to\infty}\frac1n (nc)=c>0,
\]
proving our claim.
\end{proof}

\begin{proof}[Proof of Corollary~\ref{cor:gamma-positive-finiteQ}]
Since $\Qcal$ is finite, it follows that by Proposition~\ref{prop:finite-Q-gap} and Theorem~\ref{thm:composite-maxrate} that
\[
\Gamma(\Qcal,\Pcal)=\overline d_{\mathrm{wc}}=\limsup_{n\to\infty}\frac1n a_n(\Qcal).
\]
Note that by definition, $a_n(\Qcal)=\min_{1\le i\le M}a_n(Q_i)$. Thus it follows that
\begin{equation}\label{eq:finiteQ=gamma-min}
\Gamma(\Qcal,\Pcal)=\limsup_{n\to\infty}\frac1n \min_{1\le i\le M} a_n(Q_i)=\limsup_{n\to\infty}\min_{1\le i\le M} \frac{a_n(Q_i)}{n}.   
\end{equation}
Now, for each $i\in\{1,\dots,M\}$, by Lemma~\ref{lem:an-superadditive-sec} we have that $a_n(Q_i)/n \to d_{\star}(Q_i,\Pcal)$. Note that the minimum of finitely many convergence real sequences will converge to the minimum of the limits necessarily. So by \eqref{eq:finiteQ=gamma-min} we have that
\begin{equation}\label{eq:infKL-equiv}
\Gamma(\Qcal,\Pcal)=\min_{1\le i\le M} d_{\star}(Q_i,\Pcal)=\inf_{Q\in\Qcal}d_{\star}(Q,\Pcal).    
\end{equation}

So by definition of $d_{\star}(Q,\Pcal)$ we get that
\[
\Gamma(\Qcal,\Pcal)=\inf_{Q\in\Qcal}\lim_{n\to\infty}\frac1n \KLinf(Q^n,(\Pcal^n)^{\circ\circ}),
\]
and thus the equivalence follows, completing our proof.
\end{proof}

\begin{proof}[Proof of Corollary~\ref{cor:gamma-positive-finiteQ-wlsc}]
Our rate identity shows that pointwise weak lsc implies that for each $Q\in\Qcal$, $d_{\star}(Q,\Pcal)=\KLinf(Q,\Pcal)$. So together with Corollary~\ref{cor:gamma-positive-finiteQ} it follows that,
\[
\Gamma(\Qcal,\Pcal)=\inf_{Q\in\Qcal}d_{\star}(Q,\Pcal)=\inf_{Q\in\Qcal}\KLinf(Q,\Pcal).
\]
So, it follows that the strict positivity equivalence holds, completing our proof.
\end{proof}

\begin{proof}[Proof of Theorem~\ref{thm:uniform-power-one-compact}]
Thanks to Corollary~\ref{cor:uniform-cover-and-mix}, it follows that for every $\varepsilon\in (0,d_{\mathrm{min}})$, there exist deterministic times $t_k\uparrow \infty$ and a single test supermartingale $(E_{t_k})$ such that for every $Q\in\Qcal$ wp 1 under $Q^{\infty}$,
\[
\liminf_{k\to\infty}\frac{1}{t_k}\log E_{t_k}\ge d(Q)-\varepsilon.
\]
Recall that $d(Q)=\KLinf(Q,\Pcal)$. Now since $d(Q)\ge d_{\mathrm{min}}$ for every $Q\in\Qcal$, we immediately have that $Q^{\infty}$ almost surely for every $Q\in\Qcal$,
\[
\liminf_{k\to\infty}\frac{1}{t_k}\log E_{t_k} \ge d_{\min}-\varepsilon.
\]
The expected-log bound follows from the second conclusion of Corollary~\ref{cor:uniform-cover-and-mix}. Now take some $\alpha\in(0,1)$ and $\tau_\alpha$ as we defined. Since $d_{\mathrm{min}}-\varepsilon>0$, thanks to our almost sure lower bound we have that $E_{t_k}\to\infty$ wp 1 under $Q^{\infty}$ for every $Q\in\Qcal$. Therefore it follows that for every $Q\in\Qcal$, $Q^{\infty}(\tau_\alpha<\infty)=1$. All we need to now is show the level-$\alpha$ correctness under the null. $(E_{t_k})$ being a test supermartingale entails it is also an $e$-process. So for every $P\in\Pcal$, $\E_{P^\infty}[E_{\tau_\alpha}]\le 1$. Now, on the event $\{\tau_\alpha < \infty\}$, we have that $E_{\tau_\alpha}\ge 1/\alpha$. So it follows that
\[
1\ge \E_{P^\infty}[E_{\tau_\alpha}]
\ge \frac{1}{\alpha}P^\infty(\tau_\alpha<\infty).
\]
Therefore, $P^{\infty}(\tau_\alpha<\infty)\le \alpha$ for every $P\in\Pcal$, completing the proof.
\end{proof}

\begin{proof}[Proof of Corollary~\ref{cor:uniform-power-one-finite}]
Since $\Qcal$ is finite, thanks to both Proposition~\ref{prop:finite-Q-gap} and Theorem~\ref{thm:composite-maxrate} we have that,
\[
\Gamma(\Qcal,\Pcal)=\limsup_{n\to\infty}\frac1n a_n(\Qcal),
\]
where $a_n(\Qcal)=\min_{1\le i\le M} a_n(Q_i)$ here. Now, $\Pcal$ is weakly compact. Thus from Theorem~\ref{thm:main} applied to each $Q_i$ we have that
\[
\frac{1}{n}a_n(Q_i)\longrightarrow \KLinf(Q_i,\Pcal).
\]
Then if we take the minimum over these finitely many indices we have that
\[
\Gamma(\Qcal,\Pcal)=\min_{1\le i\le M}\KLinf(Q_i,\Pcal) =\inf_{Q\in\Qcal}\KLinf(Q,\Pcal).
\]
Hence, the equivalence is shown. The existence of the uniformly power-one test follows immediately, thanks to Theorem~\ref{thm:uniform-power-one-compact}.
\end{proof}

\begin{proof}[Proof of Theorem~\ref{thm:two-sufficient-powerone}]
The second claim of the theorem is the result of \parencite[Theorem~1]{ram2026powersequentialtestsexist}. So it remains for us to prove the first claim. Assume that $\Gamma(\Qcal,\Pcal)>0$. Thanks to Theorem~\ref{thm:composite-maxrate}, we have
\[
\Gamma(\Qcal,\Pcal)=\overline d_{\mathrm{rob}}=\limsup_{n\to\infty}\frac1n\left(\sup_{E\in(\Pcal^n)^\circ}\inf_{Q\in\Qcal}\E_{Q^n}[\log E]\right).
\]
Therefore, there exist some $m\in\mathbb N$, where for some $c>0$ and an $m$-sample $e$-variable $E^{\star}\in(\Pcal^m)^{\circ}$ such that
\[
\inf_{Q\in\Qcal}\E_{Q^m}[\log E^{\star}]\ge mc>0.
\]
Now set $t_k:=km$. Take $t_0:=0$. Define the blockwise process $S$ where $S_{t_0}:=1$ and for $k\ge 1$,
\[
S_{t_k}:=\prod_{j=1}^kE^\star(X_{(j-1)m+1},\ldots,X_{jm}).
\]
For each $P\in\Pcal$, the blocks are i.i.d.\ with law $P^m$. And since $E^{\star}\in(\Pcal^m)^{\circ}$, $\E_{P^m}[E^{\star}]\le1$. Therefore,
\[
\E_{P^{\infty}}[S_{t_k}\mid \mathcal F_{t_{k-1}}]= S_{t_{k-1}}\E_{P^m}[E^\star]\le S_{t_{k-1}},
\]
proving that $(S_{t_k})_{k\ge 0}$ is a blockwise test supermartingale for $\Pcal$. Let us now consider some $Q\in\Qcal$. Under $Q^{\infty}$, the block log-increments
\(
Y_j:=\log E^\star(X_{(j-1)m+1},\ldots,X_{jm})
\)
are i.i.d.\ and satisfy $\E_{Q^{\infty}}[Y_1]=\E_{Q^m}[\log E^{\star}]\ge mc>0$. By the extended strong law of large numbers (i.e.\ allowing the mean to be $+\infty$) it holds that $Q^{\infty}$ almost surely
\[
\frac1k\sum_{j=1}^kY_j \longrightarrow \E_{Q^m}[\log E^\star] \in(0,\infty].
\]
Therefore, $\frac{\log S_{t_k}}{t_k}=\frac{1}{km}\sum_{j=1}^k Y_j$ will have strictly positive $\liminf$ with probability 1 under $Q^{\infty}$. In particular for every $Q\in\Qcal$, $S_{t_k}\to\infty$ almost surely. Now define
\[
\tau_{\alpha}:=\inf\{t_k:k\ge 0,\, S_{t_k}\ge 1/\alpha\}.
\]
Then, it holds that $\tau_\alpha$ is a $\bbF$-stopping rule, and since $S_{t_k}$ diverges with probability one, this implies that for every $Q\in\Qcal$. $Q^{\infty}(\tau_{\alpha}<\infty)=1$. Now $\alpha$-correctness is immediate from Ville's inequality applied to the blockwise test supermartingale. That is, for every $P\in\Pcal$, we have that $P^{\infty}(\tau_{\alpha}<\infty)\le\alpha$. Equivalently, we can apply optional stopping to $\tau_{\alpha}\wedge t_N$, giving us
\[
1\ge \E_{P^{\infty}}[S_{\tau_{\alpha}\wedge t_N}]\ge \frac{1}{\alpha}P^{\infty}(\tau_{\alpha}\le t_N).
\]
Then letting $N\to\infty$ gives us $P^{\infty}(\tau_{\alpha}<\infty)\le \alpha$. So it follows that $\tau_{\alpha}$ is a level-$\alpha$ pointwise power-one test, competing the proof.
\end{proof}

\end{document}